\documentclass[a4paper,twoside,12pt]{amsart}

\usepackage{ae}
\usepackage{amsmath}
\usepackage{amssymb}
\usepackage{diagrams}
\usepackage{url}

\newtheorem{thm}{Theorem}[section]
\newtheorem{prop}[thm]{Proposition}
\newtheorem{lem}[thm]{Lemma}
\newtheorem{cor}[thm]{Corollary}

\theoremstyle{definition}
\newtheorem{defi}[thm]{Definition}
\newtheorem{example}[thm]{Example}
\newtheorem{conjecture}[thm]{Conjecture}
\newtheorem{question}[thm]{Question}

\theoremstyle{remark}

\newtheorem{remark}[thm]{Remark}

\newenvironment{packed_enum}{
\begin{enumerate}
  \setlength{\itemsep}{1pt}
  \setlength{\parskip}{0pt}
  \setlength{\parsep}{0pt}
}{\end{enumerate}}

\DeclareMathAlphabet{\mathscr}{OT1}{pzc}{m}{it}

\def\jac{\textup{jac}\,}

\def\romanenum{\renewcommand{\labelenumi}{\textup{(}\roman{enumi}\textup{)}}}

\def\Bl{\textup{Bl}}

\def\Spec{\textup{Spec }}
\def\Spf{\textup{Spf }}
\def\Max{\textup{Max }}

\def\ker{\textup{ker }}

\def\Hom{\textup{Hom}}

\def\O{\mathcal{O}}

\def\N{\mathbb{N}}

\def\m{\mathfrak{m}}
\def\a{\mathfrak{a}}

\def\n{\mathfrak{n}}

\def\P{\mathbb{P}}
\def\p{\mathfrak{p}}

\def\id{\textup{id}}

\def\an{\textup{an}}
\def\rig{\textup{rig}}
\def\Rig{\textup{Rig}}
\def\FF{\textup{FF}}
\def\sRig{\textup{uRig}}
\def\uRig{\textup{uRig}}
\def\D{\mathbb{D}}
\def\srig{\textup{urig}}
\def\urig{\textup{urig}}

\def\Sp{\textup{Sp\,}}

\def\sSp{\textup{sSp\,}}
\def\sp{\textup{sp}}
\def\sI{\mathcal{I}}

\def\sK{\mathcal{K}}

\def\sT{\mathcal{T}}

\def\fU{\mathfrak{U}}
\def\fV{\mathfrak{V}}

\def\fX{\mathfrak{X}}
\def\fY{\mathfrak{Y}}

\newcommand{\ul}[1]{\underline{#1}}

\def\rad{\textup{rad}}

\def\sup{\textup{sup}}

\def\sF{\mathcal{F}}
\def\sG{\mathcal{G}}

\def\Gauss{\textup{Gauss}\,}

\def\fC{\mathfrak{C}}
\def\parent{\textup{par}}
\def\children{\textup{ch}}
\def\leaves{\textup{lv}}
\def\subtree{\textup{subt}}

\def\sr{\textup{ur}}
\def\ur{\textup{ur}}
\def\r{\textup{r}}
\def\fr{\mathfrak{r}}

\def\Zar{\textup{Zar}}

\def\aux{\textup{aux}}

\def\comp{\textup{comp}}
\def\Spa{\textup{Spa}}

\def\sG{\mathscr{G}}

\def\Specns{\textup{Spec}}
\def\sSpns{\textup{sSp}}
\def\Spfns{\textup{Spf}}

\def\phi{\varphi}

\romanenum

\begin{document}

\title{\textsc{Uniformly rigid spaces}}

\author{Christian Kappen}
\email{christian.kappen@uni-due.de}

\address{
Institut für Experimentelle Mathematik\\
Ellernstrasse 29, 45326 Essen}

\begin{abstract}
We define a new category of non-archimedean analytic spa\-ces over a complete discretely valued field, which we call \emph{uniformly rigid}. It extends the category of rigid spaces, and it can be described in terms of bounded functions on products of open and closed polydiscs. We relate uniformly rigid spaces to their associated classical rigid spaces, and we transfer various constructions and results from rigid geometry to the uniformly rigid setting. In particular, we prove an analog of Kiehl's patching theorem for coherent ideals, and we define the uniformly rigid generic fiber of a formal scheme of formally finite type. This uniformly rigid generic fiber is more intimately linked to its model than the classical rigid generic fiber obtained via Berthelot's construction.
\end{abstract}

\maketitle

\section{Introduction}

Let $K$ be a non-archimedean field, and let $R$ be its valuation ring, equipped with the valuation topology. Grothendieck had suggested that rigid spaces over  $K$ should be viewed as generic fibers of formal schemes of \emph{topologically finite} (tf) type over $R$, that is, of formal schemes which are locally isomorphic to formal spectra of quotients of strictly convergent power series rings in finitely many variables
\[
R\langle T_1,\ldots,T_n\rangle\;.
\] 
He envisaged that rigid spaces should, in a suitable sense, be obtained from these formal schemes by tensoring over $R$ with $K$. In accordance with this point of view, there is a generic fiber functor
\[
\rig\,:\,\left(\begin{array}{cc}\textup{formal $R$-schemes}\\\textup{of locally tf type}\end{array}\right)\rightarrow(\textup{rigid $K$-spaces})
\]
characterized by the property that it maps affine objects to affinoid spaces such that, on the level of functions, it corresponds to the extension of scalars functor $\cdot\otimes_RK$. This functor was more closely studied first by Raynaud and later by Bosch and Lütkebohmert; they proved that it induces an equivalence
between the category of quasi-paracompact and quasi-separated rigid $K$-spaces and the category of quasi-paracompact admissible formal $R$-schemes, localized with respect to the class of admissible blowups, cf.\ \cite{Raynaud_formalrigid}, \cite{frg1} and \cite{bosch_frgnotes} Theorem 2.8/3. 


From now on, let us assume that the absolute value on $K$ is discrete, so that $R$ is noetherian. Berthelot has extended the generic fiber functor to the class of formal $R$-schemes of locally \emph{formally finite} (ff) type, which are locally isomorphic to formal spectra of topological quotients of mixed formal power series rings in finitely many variables
\[
R[[S_1,\ldots,S_m]]\langle T_1,\ldots,T_n\rangle\;,
\]
where an ideal of definition is generated by the maximal ideal of $R$ and by the $S_i$, cf.\ \cite{rapoport-zink} Section 5.5, \cite{berthelot_rigcohpreprint} 0.2 and \cite{dejong_crystalline} 7.1--7.2. This extension of $\rig$ is characterized by the property that it maps admissible blowups to isomorphisms, where a blowup is called admissible if it is defined by an ideal that locally contains a power of a uniformizer of $R$, cf.\ \cite{temkin_desing} 2.1. The extended $\rig$ functor no longer maps affine formal schemes to affinoid spaces; for example, the generic fiber of the affine formal $R$-scheme $\Spf R[[S]]$ is the open rigid unit disc over $K$, which is not quasi-compact.

While Raynaud's generic fiber functor is precisely described in terms of admissible blowups, Berthelot's extended generic fiber functor is less accessible: for example, let us consider an unbounded function $f$ on the open rigid unit disc $\D^1_K$. The resulting morphism $\phi$ from $\D^1_K$ to the rigid projective line does not extend to models of ff type; indeed, the domain of a model of $\phi$ cannot be quasi-compact, for otherwise $f$ would be bounded. In particular, there exists no admissible blowup of $\Spf R[[S]]$ admitting an extension of $\phi$, and the schematic closure of the graph of $\phi$ in the fibered product of $\Spf R[[S]]$ and $\P^1_R$ over $\Spf R$ does not exist. This phenomenon presents a serious obstacle if one tries for example to develop a theory of Néron models of ff type. 

The main object of this article is to present a new category of non-archi\-mede\-an analytic spaces, the category of \emph{uniformly rigid spaces}, which are better adapted to formal schemes of locally ff type than Tate's rigid analytic spaces. Intuitively speaking, uniformly rigid spaces and their morphisms are described in terms of \emph{bounded} functions on finite products of open and closed unit discs. Like rigid $K$-spaces, uniformly rigid $K$-spaces are locally ringed G-topological $K$-spaces, where the letter G indicates that the underlying set of physical points is not equipped with a topology, but with a Grothendieck topology. 
Let us give a brief overview of our definitions and results.

We say that a $K$-algebra is \emph{semi-affinoid} if it is obtained from an $R$-algebra of ff type via the extension of scalars functor $\cdot\otimes_RK$. In other words, semi-affinoid $K$-algebras are quotients of $K$-algebras of the form
\[
(R[[S_1,\ldots,S_m]]\langle T_1,\ldots,T_n\rangle)\otimes_RK\;.
\]
We define the category of semi-affinoid $K$-spaces as the opposite of the category of semi-affinoid $K$-algebras, where a morphism of semi-affinoid $K$-algebras is simply a $K$-algebra homomorphism. Semi-affinoid $K$-spaces play the role of 'building blocks' for uniformly rigid $K$-spaces,
such that we effectively implement Grothendieck's original point of view in the ff type situation. Semi-affinoid $K$-algebras can be studied via  the universal properties of the free semi-affinoid $K$-algebras, which we establish in Theorem \ref{freesemaffthm}.


We define a G-topology on the category of semi-affinoid $K$-spaces equipped with its physical points functor by considering compositions of admissible blowups, completion morphisms and open immersions on flat affine models of ff type, cf.\ Definitions \ref{semaffsubdomdefi} and \ref{maintauxdefi}. These formal-geometric constructions define semi-affinoid subdomains, which may be regarded as nested rational subdomains involving strict or non-strict inequalities in semi-affinoid functions. In contrast to the classical rigid case, we cannot avoid nested constructions; this is essentially due to the fact that admissible blowups defined on open formal subschemes need not extend, cf.\ Remark \ref{nestedrequrem}. 
Just like in rigid geometry, the disconnected covering of the closed semi-affinoid unit disc $\sSp K\langle S\rangle$ by the open semi-affinoid unit disc $\sSp K[[S]]$ and the semi-affinoid unit circle $\sSp K\langle S,S^{-1}\rangle$ is \emph{not} admissible in the uniformly rigid G-topology, cf.\ Example \ref{nonadmdisccovex}. In particular, contrary to the rigid-analytic situation, finite coverings of semi-affinoid spaces by semi-affinoid subdomains need not be admissible.

Using methods from formal geometry, we prove a uniformly rigid acyclicity theorem, which in particular implies the following:
\begin{thm}[\ref{acyclicitytheorem}]
The presheaf of semi-affinoid functions is a sheaf for the uniformly rigid G-topology.
\end{thm}

The resulting functor from the category of semi-affinoid $K$-spaces to the category of locally G-ringed $K$-spaces is fully faithful; hence global uniformly rigid $K$-spaces can be defined, cf.\ Definition \ref{srigspacedefi}. They can be constructed by means of standard glueing techniques; this is possible because uniformly rigid spaces satisfy the properties (G$_0$)--(G$_2$) listed in \cite{bgr} p.\ 339. It follows that the category of uniformly rigid $K$-spaces admits fibered products and that there is a natural generic fiber functor $\urig$ from the category of formal $R$-schemes of locally ff type to the category of uniformly rigid $K$-spaces. The final picture can be described as follows:
\begin{thm}[Section \ref{compfuncsec}]
Let $\Rig_K$, $\uRig_K$ and $\FF_R$ denote the categories of rigid $K$-spaces, of uniformly rigid $K$-spaces and of formal $R$-schemes of locally ff type respectively. Let moreover $\Rig'_K\subseteq\Rig_K$ be the full subcategory of rigid spaces that are quasi-paracompact and quasi-separated. There is a diagram of functors
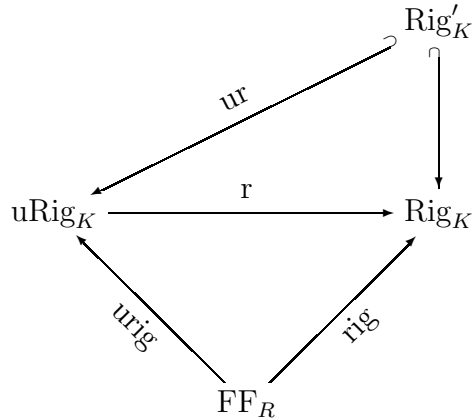
\begin{diagram}
&&&&&\Rig'_K&\\
&&&&\ldInto(4,2)<\ur&\dInto&\\
&\uRig_K&&\rTo^\r&&\Rig_K&\\
{\textup{ }}&&\luTo<\urig&&\ruTo>\rig&&&\\
&&&\FF_R&&&
\end{diagram}
commuting up to isomorphism, where 
\begin{enumerate}
\item the functor $\r$ is defined by applying the functor $\rig$ locally to models of ff type, where 
\item the functor $\ur$ is defined by applying $\urig$ to a global Raynaud-type model of locally tf type
\end{enumerate}
and where the following holds:
\begin{enumerate}
\item The functor $\ur$ is a full embedding.
\item The functor $\r$ is faithful, yet not fully faithful. 
\item  For each $X\in\uRig_K$, there is a \emph{comparison morphism} $\comp_X:X^r\rightarrow X$ that is final among all morphisms of locally G-ringed $K$-spaces from rigid $K$-spaces to $X$; it is a bijection on physical points, and it induces isomorphisms of stalks.
\end{enumerate}
\end{thm}

For $X\in\uRig_K$, we say that $X^\r$ is the underlying rigid space of $X$. Conversely, for $Y\in\Rig'_K$ we say that $Y^\ur$ is the Raynaud-type uniformly rigid structure on $Y$. Via the comparison morphisms, uniformly rigid spaces and their underlying rigid spaces are locally indistinguishable; we may thus view a uniformly rigid space as a rigid space equipped with an additional global uniform structure which is encoded in terms of a coarser G-topology and a smaller sheaf of analytic functions. Let us point out that the open rigid unit disc carries two canonical uniform structures, the one given by a Raynaud model of locally tf type and the one given by the canonical affine model $\Spf R[[S]]$ of ff type. The corresponding uniformly rigid spaces are distinct, since one is quasi-compact while the other one is not quasi-compact. The fact that $\r$ is not fully faithful is seen by the example of an \emph{unbounded} function $f$ on the rigid open unit disc which we considered above: The rigid-analytic morphism $\phi$ defined by $f$ does not extend to a morphism of uniformly rigid spaces from $(\Spf R[[S]])^\urig$ to $(\P_K^{1,\an})^\ur$. 


In Section \ref{cohmodsec}, we study coherent modules on uniformly rigid $K$-spaces. We prove the existence of schematic closures of coherent submodules, cf.\ Theorem \ref{frameembthm}. Using the resulting models of coherent ideals, we prove the following analog of Kiehl's theorem A in rigid geometry, cf.\ \cite{kiehlab}: 

\begin{thm}[\ref{frameembassoccor}]
Coherent ideals on semi-affinoid spaces are associated to their ideals of global functions.
\end{thm}
In particular, closed uniformly rigid subspaces are well-behaved, cf.\ Proposition \ref{closedimprop}. Using fibered products and closed uniformly rigid subspaces, we define the notion of separatedness for uniformly rigid $K$-spaces, and we define the graph of a morphism $f:Y\rightarrow X$ of uniformly rigid $K$-spaces whose target is separated, cf.\ Section \ref{separatedsection}. Using Theorem \ref{frameembthm}, we show that if $\fX$ and $\fY$ are flat formal $R$-schemes of locally ff type such that $\fX^\urig$ is separated and if $f:\fY^\urig\rightarrow\fX^\urig$ is a morphism of uniformly rigid generic fibers, then the schematic closure of the graph of $f$ in $\fY\times\fX$ exists. As we have noted above, the corresponding statement is false if $\urig$ is replaced by Berthelot's generic fiber functor $\rig$.

Semi-affinoid algebras and some associated locally G-ringed $K$-spaces have already been studied by Lipshitz and Robinson in \cite{lipshitz_robinson}, where the terminology \emph{quasi-affinoid} is used. The approach in \cite{lipshitz_robinson} includes the situation where $R$ is not discrete and where the machinery of locally noetherian formal geometry is not available. However, no global theory is developed in \cite{lipshitz_robinson}, and the connection to formal geometry is not discussed. The proof of Theorem \ref{freesemaffthm} in the case of a possibly non-discrete valuation, given in \cite{lipshitz_robinson} I.5.2.3, is technically more involved, and it relies upon methods different from ours.  The definition of the G-topology in \cite{lipshitz_robinson} III.2.3.2 is less explicit than our definition, so that a deep quantifier elimination theorem \cite{lipshitz_robinson} II. Theorem 4.2 is needed in order to prove an acyclicity theorem. Our approach avoids quantifier elimination. 

It is unclear how to reflect uniformly rigid structures on the level of Ber\-ko\-vich's analytic spaces or on the level of Huber's analytic adic spaces, cf.\ Section \ref{berthhubercompsec}. Semi-affinoid $K$-algebras are equipped with unique $K$-Banach algebra structures, so that one may consider their valuation spectra. For instance, the spectrum $M(R[[S]]\otimes_RK)$ is the closure of the Berkovich open unit disc within the Berkovich closed unit disc; it is obtained by adding the Gauss point. However, inclusions of semi-affinoid subdomains need not induce injective maps of valuation spectra, so the formation of the valuation spectrum does not globalize. This corresponds to the fact that in the ff type situation, the functor $\cdot\otimes_RK$ does not commute with complete localization. Nonetheless, we suggest that a uniformly rigid $K$-space $X$ should be viewed as a compactification of its underlying rigid space $X^\r$. This point of view might be useful in order to obtain a better understanding of the quasi-compactifications considered in \cite{strauch_deformation} 3.1 and in \cite{huber_finiteness_ii} 3; it should be further developed within the framework of topos theory. We propose the study of the uniformly rigid topos as a topic for future research. 

The author has used uniformly rigid spaces in his doctoral thesis \cite{mythesis}, in order to lay the foundations for a theory of formal Néron models of locally ff type. The search for such a theory was strongly motivated by work of C.-L. Chai \cite{bisecartpre}, who had suggested that Néron models of ff type could be used to study the base change conductor of an abelian variety with potentially multiplicative reduction over a local field. Chai and the author are currently working on further developing the methods of \cite{bisecartpre} within the framework of uniformly rigid spaces.

The present paper contains parts of the first chapter of the author's dissertation \cite{mythesis}. He would like to express his gratitude to his thesis advisor Siegfried Bosch. Moreover, he would like to thank Brian Conrad, Ofer Gabber, Ulrich Görtz, Philipp Hartwig, Simon Hüsken, Christian Wahle and the referee for helpful discussions and comments, and he would like to thank the Massachusetts Institute of Technology for its hospitality. This work was financially supported by the German National Academic Foundation, by the Graduiertenkolleg Analytic Topology and Metageometry of the University of Münster and by the Hamburger Stiftung für internationale Forschungs- und Studienvorhaben; the author would like to extend his gratitude to these institutions.

\tableofcontents

\section{Uniformly rigid spaces}

Let $R$ be a discrete valuation ring with residue field $k$ and fraction field $K$, and let $\pi\in R$ be a uniformizer.

\subsection{Formal schemes of formally finite type}\label{fftypesec}

A morphism of locally noetherian formal schemes is said to be of locally \emph{formally finite} (ff) type if the induced morphism of smallest subschemes of definition is of locally finite type. Equivalently, any induced morphism of subschemes of definition is of locally finite type. A morphism of locally noetherian formal schemes is called of ff type if it is of locally ff type and quasi-compact. If $A$ is a noetherian adic ring and if $B$ is a noetherian adic topological $A$-algebra, then $\Spf B$ is of ff type over $\Spf A$ if and only if $B$ is a topological quotient of a mixed formal power series ring $A[[S_1,\ldots,S_m]]\langle T_1,\ldots,T_n\rangle$, where $A[[S_1,\ldots,S_m]]$ carries the $\a+(S_1,\ldots,S_m)$-adic topology for any ideal of definition $\a$ of $A$, cf.\ \cite{berkvancyc2} Lemma 1.2. In this case, we say that the topological $A$-algebra $B$ is of ff type. Morphisms of locally ff type are preserved under composition, base change and formal completion. 

We say that an $R$-algebra is of \emph{formally  finite} (ff) type if it admits a ring topology such that it becomes a topological $R$-algebra of ff type in the above sense, where $R$ carries the $\pi$-adic topology. Equivalently, an $R$-algebra is of ff type if it admits a presentation as a quotient of a mixed formal power series ring, as above. If $S$ and $T$ are finite systems of variables and if $\phi:R[[S]]\langle T\rangle\rightarrow A$ is a surjection, then the $\phi$-image of $(S,T)$ will be called a formal generating system for $A$.

\begin{lem}\label{topinducedlem}
If $A$ is a topological $R$-algebra of ff type, then the biggest ideal of definition of $A$ coincides with the Jacobson radical of $A$. Moreover, any $R$-homomorphism of topological $R$-algebras of ff type is continuous.
\end{lem}
\begin{proof}
Let $\a$ denote the biggest ideal of definition of $A$; then $\a$ is contained in every maximal ideal of $A$ since $A$ is $\a$-adically complete. On the other hand, $A/\a$ is a Jacobson ring since it is of finite type over the residue field $k$ of $R$; it follows that $\a$ coincides with the Jacobson radical of $A$, as claimed. In particular, the topology on $A$ is determined by the ring structure of $A$. Let now $A\rightarrow B$ be a homomorphism of $R$-algebras of ff type; by what we have seen so far, it suffices to see that $\phi$ is continuous for the Jacobson-adic topologies. However, for any maximal ideal $\n\subseteq B$, the preimage $\m:=\n\cap A$ of $\n$ in $A$ is maximal, since $k\subseteq A/\m\subseteq B/\n$, where $B/\n$ is a finite field extension of $k$ because the quotient $B/\jac B$ is of finite type over $k$.
\end{proof}

In particular, the topology on $A$ an be recovered from the ring structure on $A$, and the category of $R$-algebras of ff type is canonically equivalent to the category of \emph{topological} $R$-algebras of ff type.
Lemma \ref{topinducedlem} implies that the category of $R$-algebras of ff type admits amalgamated sums $\hat{\otimes}$. 

\subsection{Semi-affinoid algebras}

We define semi-affinoid $K$-algebras as the generic fibers of $R$-algebras of ff type, and we define the category of semi-affinoid $K$-spaces as the dual of the category of semi-affinoid $K$-algebras:

\begin{defi}\label{saffalgdefi}
Let $A$ be a $K$-algebra.
\begin{packed_enum}
\item An \emph{$R$-model} of $A$ is an $R$-subalgebra $\ul{A}\subseteq A$ such that the natural homomorphism $\ul{A}\otimes_RK\rightarrow A$ is an isomorphism.
\item The $K$-algebra $A$ is called \emph{semi-affinoid} if it admits an $R$-model of ff type.
\item A homomorphism of semi-affinoid $K$-algebras is a homomorphism of underlying $K$-algebras.
\item The category of semi-affinoid $K$-spaces is the dual of the category of semi-affinoid $K$-algebras. If $A$ is a semi-affinoid $K$-algebra, we write $\sSp A$ to denote the corresponding semi-affinoid $K$-space, and if $\phi:\sSp B\rightarrow\sSp A$ is a morphism of semi-affinoid $K$-spaces, we write $\phi^*$ to denote the corresponding $K$-algebra homomorphism.
\end{packed_enum}
\end{defi}

By Definition \ref{saffalgdefi} ($i$) above, any $R$-model of a $K$-algebra is flat over $R$.

There exists no general analog of the Noether normalization theorem for semi-affinoid $K$-algebras,
cf.\ \cite{lipshitz_robinson} I.2.3.5. However, if $A$ is a semi-affinoid $K$-algebra admitting a \emph{local} $R$-model of ff type, then there exist finitely many variables $S_1,\ldots,S_m$ and a finite $K$-monomorphism
\[
R[[S_1,\ldots,S_m]]\otimes_RK\hookrightarrow A\;.
\]
Indeed, if $\ul{A}$ is a local $R$-model of ff type for $A$ with maximal ideal $\m$ and if $s_0,\ldots,s_m$ is a system of parameters for $\ul{A}$ such that $s_0=\pi$, then there exists a unique continuous $R$-homomorphism $\phi\colon R[[S_1,\ldots,S_m]]\rightarrow\ul{A}$ sending $S_i$ to $s_i$, for $1\leq i\leq m$. If $\fr$ denotes the maximal ideal of $R[[S_1,\ldots,S_m]]$, then $\ul{A}/\fr\ul{A}$ is $k$-finite because $\ul{A}/\m\ul{A}$ is $k$-finite and because $\fr$ is $\m$-primary.  By the formal version of Nakayama's Lemma, cf.\ \cite{eisenbudca} Ex.\ 7.2, it follows that $\phi$ is finite; here we use that $R[[S_1,\ldots,S_m]]$ is $\fr$-adically complete and that $\ul{A}$ is $\fr$-adically separated. Since $R[[S_1,\ldots,S_m]]$ and $\ul{A}$ have the same dimension, it follows that $\phi$ is finite, so we obtain the desired finite monomorphism by extending scalars from $R$ to $K$.

\subsubsection{The specialization map}\label{specmapsec}
For the following statement, cf.\ \cite{dejong_crystalline} 7.1.9: 
\begin{lem}\label{specmaplem}
Let $A$ be a semi-affinoid $K$-algebra, and let $\ul{A}\subseteq A$ be an $R$-model of ff type. If $\m$ is a maximal ideal in $A$, then
\[
\sp_{\ul{A}}(\m)\,:=\sqrt{(\ul{A}\cap\m)+\pi\ul{A}}
\]
is a maximal ideal in $\ul{A}$, and $A/\m$ is a finite extension of $K$.
\end{lem}
\begin{proof}
Let us write $\p\mathrel{\mathop:}=\m\cap\ul{A}$; then $(\ul{A}/\p)_\pi=A/\m$ is a field, and by the Artin-Tate Theorem \cite{egaiv} 0.16.3.3 it follows that $\ul{A}/\p$ is a semi-local ring of dimension $\leq 1$. Moreover, $\ul{A}/\p$ is of ff type over $R$ and, hence, $\pi$-adically complete. Since $\ul{A}/\p\subseteq A/\m$ is $R$-flat and since $(\ul{A}/\p)_\pi$ is local, it thus follows from Hensel's Lemma that $(\ul{A}/\p)/\pi(\ul{A}/\p)$ is local as well, cf.\ \cite{bourbakica} III.4.6 Proposition 8. Since $\p A=\m$, the class of $\pi$ in $\ul{A}/\p$ is nonzero, and so the local noetherian ring $(\ul{A}/\p)/\pi(\ul{A}/\p)$ is zero-dimensional. Thus its quotient modulo its nilradical is a field, and it follows that the radical of $\p+\pi\ul{A}$ is maximal in $\ul{A}$, as desired.

To prove that $A/\m$ is finite over $K$, it suffices to show that $\ul{A}/\p$ is finite over $R$. Since $R$ is $\pi$-adically complete and since $\ul{A}/\p$ is $\pi$-adically separated, it thus suffices to show that $\ul{A}/(\p+\pi\ul{A})$ is finite over $k$, cf.\ \cite{eisenbudca} Ex. 7.2. The ring $\ul{A}/(\p+\pi\ul{A})$ is noetherian; hence its nilradical is nilpotent, and it thereby suffices to see that the quotient of $\ul{A}$ modulo the maximal ideal $\sqrt{\p+\pi\ul{A}}$ is $k$-finite. Since $\ul{A}$ is of ff type over $R$, since maximal ideals are open and since field extensions of finite type are finite, the desired statement follows.
\end{proof}

\begin{defi}
If $A$ is a semi-affinoid $K$-algebra, we call $|X|:=\Max A$ the set of physical points of its corresponding semi-affinoid $K$-space $X$. We will often write $X$ instead of $|X|$ if no confusion is likely to result. 
\end{defi}

\begin{remark}\label{specmaprem}
Lemma \ref{specmaplem} implies that a morphism $\phi:\sSp A\rightarrow \sSp B$ induces a map on sets of physical points such that for $R$-models of ff type $\ul{A}$ and $\ul{B}$ with $\phi^*(\ul{B})\subseteq\ul{A}$, the specialization maps $\sp_{\ul{A}}$ and $\sp_{\ul{B}}$ are compatible with respect to $\phi$ and the induced morphism $\ul{\phi}:\Spf\ul{A}\rightarrow\Spf\ul{B}$. This functoriality implies that $\sp_{\ul{A}}$ is surjective onto the set of maximal ideals in $\ul{A}$. Indeed, let $\r\subseteq \ul{A}$ be a maximal ideal, and let $\ul{A}|_\r$ denote the $\r$-adic completion of $A$; then $\Max(\ul{A}|_\r\otimes_RK)$ is nonempty, and any element in this set maps to an element in $\Max (A)$ that maps to $\r$ under $\sp_{\ul{A}}$. Let us moreover remark that for $x\in X=\sSp A$ with specialization $\n\subseteq\ul{A}$, the valuation ring of the residue field of $A$ in $x$ coincides with the integral closure of $\ul{A}_\n$ in that residue field, so that the intersection of $\ul{A}_\n$ with the valuation ideal is precisely $\n \ul{A}_\n$.
\end{remark}

\subsubsection{Power-boundedness and topological quasi-nilpotency}\label{pbsec}
Let $X$ be a se\-mi-affinoid $K$-space with corresponding semi-affinoid $K$-algebra $A$. By Lemma \ref{specmaplem}, $A/\m$ is $K$-finite for $\m\subseteq A$ maximal; hence the discrete valuation on $K$ extends uniquely to $A/\m$, so we can define $|f(x)|\in\mathbb{R}_{\geq 0}$ for any $f\in A$, $x\in X$.
\begin{defi}
An element $f\in A$ is called \emph{power-bounded} if $|f(x)|\leq 1$ for all $x\in X$. It is called \emph{topologically quasi-nilpotent} if $|f(x)|<1$ for all $x\in X$. We let $\mathring{A}\subseteq A$ denote the $R$-subalgebra of power-bounded functions, and we let $\check{A}\subseteq\mathring{A}$ denote the ideal of topologically quasi-nilpotent functions.
\end{defi}
For example, $S\in A=R[[S]]\otimes_RK$ is topologically quasi-nilpotent, while the supremum of the absolute values $|S(x)|$, with $x$ ranging over $X$, is equal to $1$. Thus we see that the classical maximum principle fails for semi-affinoid $K$-algebras. However, the maximum principle holds if we let $x$ vary in the Berkovich spectrum $M(A)$ of $A$, where $A$ is equipped with its unique $K$-Banach algebra topology, cf.\ Section \ref{berthhubercompsec}. Indeed, this follows trivially from the fact that $M(A)$ is compact.

\begin{remark}
If $A$ is a non-reduced semi-affinoid $K$-algebra, then $\mathring{A}$ cannot be of ff type over $R$: If $f\in A$ is a nonzero nilpotent function, then $f\in\mathring{A}$ is infinitely $\pi$-divisible in $\mathring{A}$, but $R$-algebras of ff type are $\pi$-adically separated. 
\end{remark}

\begin{remark}\label{topqnilpotrem}
If $\ul{A}\subseteq A$ is an $R$-model of ff type, then $\ul{A}\subseteq\mathring{A}$, and $\check{A}\cap\ul{A}\subseteq\ul{A}$ is the biggest ideal of definition. Indeed, by Lemma \ref{topinducedlem} and its proof, the biggest ideal of definition of $\ul{A}$ is given by the Jacobson radical, and hence it suffices to observe that for any $f\in\ul{A}$ and any $x\in\sSp A$ with specialization $\n\subseteq\ul{A}$, we have $|f(x)|\leq 1$, where $|f(x)|<1$ if and only if $f\in\n$. This however is clear from the final statement in Remark \ref{specmaprem}.
\end{remark}

For the notion of \emph{normality} for formal $R$-schemes of locally ff type, we refer to the discussion in \cite{conrad_irred} 1.2, which is based on the fact that $R$-algebras of ff type are \emph{excellent}. This excellence result is a consequence of \cite{valabrega1} Proposition 7 if $R$ has equal characteristic, and it follows from \cite{valabrega2} Theorem 9 if $R$ has mixed characteristic. In the following, excellence of $R$-algebras of ff type will be used without further comments. 

The following result is fundamental:

\begin{prop}\label{normalmodelprop}
Let $A$ be a semi-affinoid $K$-algebra. If $A$ admits a \emph{normal} $R$-model of ff type, then this model coincides with $\mathring{A}$.
\end{prop}
\begin{proof}
Let $\ul{A}$ be a normal $R$-model of ff type for $A$. By \cite{dejong_crystalline} 7.1.8, we may view $A$ as a subring of the ring of global functions on $(\Spf\ul{A})^\rig$, and by \cite{dejong_crystalline} 7.4.1, \cite{dejong_crystalline_err}, $\ul{A}$ coincides with the ring of power-bounded global functions under this identification.  
\end{proof}

\begin{cor}\label{model1cor}
Let $A$ be a semi-affinoid $K$-algebra, and let $\ul{A}\subseteq A$ be an $R$-model of ff type; then the inclusion $\ul{A}\subseteq\mathring{A}$ is integral. If moreover $A$ is reduced, then this inclusion is finite.
\end{cor}
\begin{proof}
Let $\ul{\phi}\colon\ul{A}\rightarrow\ul{B}$ denote the normalization of $\ul{A}$. Then $\ul{\phi}$ is finite since $\ul{A}$ is excellent, and hence $\ul{B}$ is of ff type over $R$. Extension of scalars yields an induced homomorphism of semi-affinoid $K$-algebras  $\phi\colon A\rightarrow B$. Since $\ul{\phi}$ factors through an injective $R$-homomorphism $\ul{A}/\rad(\ul{A})\hookrightarrow\ul{B}$, since $K$ is $R$-flat and since $\rad(A)=\rad(\ul{A}) A$, we see that $\phi$ factors through an injective $K$-homomorphism $A/\rad(A)\hookrightarrow B$. By Proposition \ref{normalmodelprop}, $\ul{B}$ coincides with the ring of power-bounded functions in $B$. Let us consider a power-bounded function $f$ in $A$; then $\phi(f)\in\ul{B}$. Since $\ul{\phi}$ is finite, there exists an integral equation $P(T)\in\ul{A}[T]$ for $\phi(f)$ over $\ul{A}$. By the factorization of $\phi$ mentioned above, we conclude that $P(f)\in A$ is nilpotent. If $s\in\N$ is an integer such that $P(f)^s=0$; then $P(T)^s\in\ul{A}[T]$ is an integral equation for $f$ over $\ul{A}$. Finally, if $A$ is reduced, then $\phi$ is injective, and hence $\mathring{A}$ is an $\ul{A}$-submodule of the finite $\ul{A}$-module $\ul{B}$. Since $\ul{A}$ is noetherian, it follows that $\mathring{A}$ is a finite $\ul{A}$-module. 
\end{proof}

We immediately obtain the following:
\begin{cor}
The ring of power-bounded functions in a reduced semi-affinoid $K$-algebra is a canonical $R$-model of ff type containing any other $R$-model of ff type. 
\end{cor}

We conclude that any $R$-model of ff type can be enlarged so that it contains any given finite set of power-bounded functions:

\begin{cor}\label{model2cor}
Let $\ul{A}$ be an $R$-model of ff type in a semi-affinoid $K$-algebra $A$, and let $M\subseteq A$ be a finite set of power-bounded functions. Then the $\ul{A}$-subalgebra $\ul{A}[M]$ generated by $M$ over $\ul{A}$ is finite over $\ul{A}$ and, hence, an $R$-model of ff type for $A$.
\end{cor}
\begin{proof}
The ring extension $\ul{A}\subseteq\ul{A}[M]$ is finite since it is generated by finitely may integral elements.
\end{proof}

\subsubsection{Free semi-affinoid algebras} Using the results of Section \ref{pbsec}, we can now establish the universal properties of free semi-affinoid $K$-algebras; these are semi-affinoid $K$-algebras of the form $R[[S]]\langle T\rangle\otimes_RK$, for finite systems of variables $S$ and $T$:
\begin{thm}\label{freesemaffthm}
Let $m$ and $n$ be natural numbers. The semi-affinoid $K$-algebra $R[[S_1,\ldots,S_m]]\langle T_1,\ldots,T_n\rangle\otimes_RK$, together with the pair of tuples of functions $((S_1,\ldots,S_m),(T_1,\dots,T_n))$, is initial among all semi-affinoid $K$-algebras $A$ equipped with a pair $((f_1,\ldots,f_m),(g_1,\ldots,g_n))$ satisfying the property that the $g_j$ are power-bounded and that the $f_i$ are topologically quasi-nilpotent.
\end{thm}
\begin{proof}
Let us write $S$ and $T$ to denote the systems of the $S_i$ and the $T_j$. By Corollary \ref{model2cor}, $A$ admits an $R$-model of ff type $\ul{A}$ containing the $f_i$ and the $g_j$. By Remark \ref{topqnilpotrem}, the $f_i$ are topologically nilpotent in $\ul{A}$; hence there exists a unique $R$-homomorphism $\ul{\phi}\colon R[[S]]\langle T\rangle\rightarrow\ul{A}$ sending $S_i$ to $f_i$ and $T_j$ to $g_j$ for all $i$ and $j$, and so $\phi\mathrel{\mathop:}=\ul{\phi}\otimes_RK$ is a $K$-homomorphism with the desired properties. It remains to show that these properties determine $\phi$ uniquely. Let $\phi'\colon R[[S]]\langle T\rangle\otimes_RK\rightarrow A$ be any $K$-homomorphism sending $S_i$ to $f_i$ and $T_j$ to $g_j$ for all $i$ and $j$, and let us set $\ul{A}':=\phi'(R[[S]]\langle T\rangle)$ which is of ff type over $R$. If $\phi=\phi'$, then $\ul{A}'\subseteq\ul{A}$. On the other hand, to show that $\phi=\phi'$, it suffices to see that, after possibly enlarging $\ul{A}$, we have $\ul{A}'\subseteq\ul{A}$, in virtue of the universal property of $R[[S]]\langle T\rangle$. 
If $A$ is reduced, Corollary \ref{model1cor} says that we may set $\ul{A}$ equal to the ring of power-bounded functions in $A$; in this case the inclusion $\ul{A}'\subseteq\ul{A}$ is obvious. In the general case, we let $N$ denote the nilradical of $A$; then by what we have shown so far,
\[
\ul{A}'/(\ul{A}'\cap N)\,\subseteq\ul{A}/(\ul{A}\cap N)\quad\quad(*)
\]
within $\mathring{A}/N$. The ideal $\ul{A}'\cap N$ is finitely generated since $\ul{A}'$ is noetherian; after enlarging $\ul{A}$ using Corollary \ref{model2cor}, we may thus assume that $\ul{A}$ contains a generating system $n_1,\ldots,n_r$ of $\ul{A}'\cap N$. The inclusion $(*)$ shows that every element $a'\in\ul{A}'$ is the sum of an element $a\in\ul{A}$ and a linear combination $\sum_{i=1}^r a_i'n_i$ with coefficients $a_i'\in \ul{A}'$. Let us write the coefficients $a_i'$ in the analogous way, and let us iterate the procedure. Using the fact that the $n_i$ lie in $\ul{A}$, the only summands possibly not lying in $\ul{A}$ after $s$-fold iteration are multiples of products of the $n_i$ involving $s$ factors. Since the $n_i$ are nilpotent, these summands are zero for $s$ big enough; hence $\ul{A}'\subseteq\ul{A}$, as desired.
\end{proof}

With the universal property of the free semi-affinoid $K$-algebras at hand, we can now describe the category of semi-affinoid $K$-algebras in terms of the category of $R$-models of ff type. Let us recall that a formal blowup in the sense of \cite{temkin_desing} 2.1 is called \emph{admissible} if it can be defined by a $\pi$-adically open coherent ideal.

\begin{cor}\label{freesemaffcor}
Let $\phi:A\rightarrow B$ be a homomorphism of semi-affinoid $K$-algebras.
\begin{enumerate}
\item Let $\ul{A}_1$, $\ul{A}_2$ be $R$-models of ff type for $A$. If $\ul{A}_2$ contains a formal generating system of $\ul{A}_1$, then $\ul{A}_1$ is contained in $\ul{A}_2$.
\item An inclusion of $R$-models of ff type for $A$ corresponds to a finite admissible blowup of associated formal spectra.
\item Let $\ul{A}\subseteq A$, $\ul{B}\subseteq B$ be $R$-models of ff type such that there exists a formal generating system of $\ul{A}$ mapping to $\ul{B}$ via $\phi$. Then $\phi(\ul{A})\subseteq\ul{B}$.
\item Let $\ul{A}$ be an $R$-model of ff type for $A$. There exists an $R$-model of ff type $\ul{B}$ for $B$ such that $\phi(\ul{A})\subseteq\ul{B}$. Moreover, if $\ul{B}'$ is any $R$-model of ff type for $B$, we can choose $\ul{B}$ such that $\ul{B}'\subseteq\ul{B}$. 
\end{enumerate}
\end{cor}
\begin{proof}
To prove the first statement, let us fix a formal generating system $(f,g)$ of $\ul{A}_1$ that is contained in $\ul{A}_2$. The components of $f$ are topologically quasi-nilpotent in $A$; since $\ul{A}_2$ is an $R$-model of ff type for $A$, they are topologically nilpotent in $\ul{A}_2$. Let $\alpha:R[[S]]\langle T\rangle\rightarrow\ul{A}_1$ and $\beta:R[[S]]\langle T\rangle\rightarrow\ul{A}_2$ be the associated $R$-homomorphisms, where $\alpha$ is surjective because $(f,g)$ formally generates $\ul{A}_1$. By Theorem \ref{freesemaffthm}, $\alpha\otimes_RK$ and $\beta\otimes_RK$ coincide as homomorphisms from $R[[S]]\langle T\rangle\otimes_RK$ to $A$, so we conclude that $\ul{A}_1\subseteq\ul{A}_2$: given $a\in\ul{A}_1$, we choose an $\alpha$-preimage $a'$ of $a$; then $a=\alpha(a')=\beta(a')\in\ul{A}_2$.

To prove the second claim, let $\ul{A}_1\subseteq\ul{A}_2$ be an inclusion of $R$-models of ff type for $A$, and let $M\subseteq\ul{A}_2$ be a finite set whose elements are the components of a formal generating system for $\ul{A}_2$ over $R$. Then by Corollary \ref{model2cor}, $\ul{A}_1[M]\subseteq\ul{A}_2$ is an $R$-model of ff type for $A$ which is finite over $\ul{A}_1$. By statement ($i$), $\ul{A}_2=\ul{A}_1[M]$ and hence $\ul{A}_2$ is finite over $\ul{A}_1$. Arguing exactly as in the proof of \cite{frg1} 4.5, we see that $\ul{A}_1\subseteq\ul{A}_2$ corresponds to an admissible formal blowup.

To prove part $(iii)$, let us choose a formal generating system  $(f,g)$ of $\ul{A}$ such that the components of $\phi(f)$ and $\phi(g)$ are contained in $\ul{B}$. The components of $\phi(f)$ are topologically nilpotent in $\ul{B}$ since they are topologically quasi-nilpotent in $B$. Let $\alpha:R[[S]]\langle T\rangle\rightarrow\ul{A}$ and $\beta:R[[S]]\langle T\rangle\rightarrow\ul{B}$ be the $R$-homomorphisms defined by $(f,g)$ and $(\phi(f),\phi(g))$ respectively; then $\alpha$ is surjective, and Theorem \ref{freesemaffthm} shows that $\beta\otimes_RK$ coincides with $\phi\circ(\alpha\otimes_RK)$. As is the proof of statement ($i$), we conclude that $\phi(\ul{A})\subseteq\ul{B}$.

To prove statement ($iv$), let us choose a formal generating system  $(f,g)$ of $\ul{A}$. The components of $\phi(f)$ are topologically quasi-nilpotent, and the components of $\phi(g)$ are power-bounded in $B$. According to Corollary \ref{model2cor}, there exists an $R$-model $\ul{B}$ of ff type for $B$ containing $\ul{B}'$ and the components of $\phi(f)$ and $\phi(g)$; by statement ($iii$), $\phi(\ul{A})\subseteq\ul{B}$, as desired.
\end{proof}


We can now show that $R$-models of ff type for affinoid $K$-algebras are automatically of tf type:

\begin{cor}\label{afflattfintypecor}
Let $A$ be an affinoid $K$-algebra, and let $\ul{A}\subseteq A$ be an $R$-model of ff type. Then $\ul{A}$ is of tf type over $R$.
\end{cor}
\begin{proof}
Let $\ul{A}'$ be an $R$-model of tf type for $A$, and let $\ul{A}''$ be an $R$-model of ff type for $A$ containing both $\ul{A}$ and $\ul{A}'$; such an $\ul{A}''$ exists by Corollary \ref{freesemaffcor} ($iv$) applied to the identity on $A$. By Corollary \ref{freesemaffcor} ($ii$), $\ul{A}''$ is finite over $\ul{A}'$ and, hence, an $R$-algebra of tf type. After replacing $\ul{A}'$ by $\ul{A}''$, we may thus assume that $\ul{A}\subseteq\ul{A}'$. Again by Corollary \ref{freesemaffcor} ($ii$), this inclusion is finite. We now mimic the proof of the classical Artin-Tate Lemma: Let $a_1,\ldots,a_m$ be a system of topological generators of $\ul{A}'$ over $R$, and for each $i$ let $P_i\in\ul{A}[T]$ be an integral equation for $a_i$ over $\ul{A}$. Let $b_1,\ldots,b_n$ be the coefficients of the $P_i$ in some ordering. Since the $R$-algebra $\ul{A}$ is of ff type, it is $\pi$-adically complete; hence there exists a unique $R$-homomorphism $R\langle T_1,\ldots,T_n\rangle\rightarrow\ul{A}$ sending $T_j$ to $b_j$ for $1\leq j\leq n$. Let $\ul{B}\subseteq\ul{A}$ denote its image; then $\ul{B}$ is an $R$-algebra of tf type. Since the $a_i$ topologically generate $\ul{A}'$ over $R$, they also topologically generate $\ul{A}'$ over $\ul{B}$. The $a_i$ are, by construction, integral over $\ul{B}$; hence $\ul{A}'$ is in fact finite over $\ul{B}$. Since $\ul{B}$ is noetherian, the $\ul{B}$-submodule $\ul{A}$ of $\ul{A}'$ is finite as well, and it follows that $\ul{A}$ is of tf type as a $\ul{B}$-algebra. We conclude that $\ul{A}$ is of tf type over $R$.
\end{proof}

\subsubsection{Amalgamated sums}

\begin{prop}\label{amalgsumsprop}
The category of semi-affinoid $K$-algebras admits amalgamated sums. More precisely speaking, if $\phi_1\colon A\rightarrow B_1$ and $\phi_2\colon A\rightarrow B_2$ are homomorphisms of semi-affinoid $K$-algebras, then the colimit of the resulting diagram is represented by $(\ul{B}_1\hat{\otimes}_{\ul{A}}\ul{B}_2)\otimes_RK$, where $\ul{A}$ and the $\ul{B}_i$ are $R$-models of ff type for $A$ and the $B_i$ respectively such that $\phi(\ul{A})\subseteq\ul{B}_1,\ul{B}_2$.
%
\end{prop}
\begin{proof}
By Corollary \ref{freesemaffcor} ($iv$), we may choose $R$-models $\ul{A}$, $\ul{B}_1$ and $\ul{B}_2$ as in the statement of the proposition. Let $C$ be a semi-affinoid $K$-algebra, and for $i=1,2$ let $\tau_i\colon B_i\rightarrow C$ be a $K$-homomorphism such that $\tau_1\circ\phi_1=\tau_2\circ\phi_2$. By Corollary \ref{freesemaffcor} ($iv$), there exists an $R$-model $\ul{C}$ of ff type for $C$ such that $\tau_i(\ul{B}_i)\subseteq\ul{C}$ for $i=1,2$; we let $\ul{\tau}_i\colon\ul{B}_i\rightarrow\ul{C}$ denote the induced $R$-homomorphism. Then $\ul{\tau}_1\circ\ul{\phi}_1=\ul{\tau}_2\circ\ul{\phi}_2$, since the same holds after inverting $\pi$ and since $\pi$ is not a zero divisor in $\ul{A}$. By the universal property of the complete tensor product in the category of $R$-algebras of ff type, there exists a unique $R$-homomorphism $\ul{\tau}\colon\ul{B}_1\hat{\otimes}_{\ul{A}}\ul{B}_2\rightarrow\ul{C}$ such that $\ul{\tau}_i=\ul{\tau}\circ\ul{\sigma}_i$ for $i=1,2$, where $\ul{\sigma}_i:\ul{B}_i\rightarrow\ul{B}_1\hat{\otimes}_{\ul{A}}\ul{B}_2$ is the $i$th coprojection. Setting $\tau\mathrel{\mathop:}=\ul{\tau}\otimes_RK$ and $\sigma_i:=\ul{\sigma}_i\otimes_RK$, we obtain $\tau_i=\tau\circ\sigma_i$ for $i=1,2$. We must show that $\tau$ is uniquely determined by this property. Let 
\[
\tau'\colon(\ul{B}_1\hat{\otimes}_{\ul{A}}\ul{B}_2)\otimes_RK\rightarrow C
\]
be any $K$-homomorphism satisfying $\tau_i=\tau'\circ\sigma_i$ for $i=1,2$. By Corollary \ref{freesemaffcor} ($iv$), there exists an $R$-model $\ul{C}'$ of ff type for $C$ containing $\ul{C}$ such that $\tau'$ restricts to an $R$-morphism $\ul{\tau}'\colon\ul{B}_1\hat{\otimes}_{\ul{A}}\ul{B}_2\rightarrow\ul{C}'$; then $\tau'=\ul{\tau}'\otimes_RK$. It suffices to show that $\ul{\tau}'$ coincides with $\ul{\tau}$ composed with the inclusion $\ul{\iota}:\ul{C}\subseteq\ul{C}'$. For $i=1,2$, the compositions $\ul{\tau}'\circ\ul{\sigma}_i$ and $\ul{\iota}\circ\ul{\tau}\circ\ul{\sigma}_i$ coincide after inverting $\pi$, hence they coincide because $\pi$ is not a zero divisor in $\ul{B}_i$, for $i=1,2$. The universal property of $(\ul{B}_1\hat{\otimes}_{\ul{A}}\ul{B}_2,\ul{\sigma}_1,\ul{\sigma}_2)$ implies that $\ul{\tau}'=\ul{\iota}\circ\ul{\tau}$, as desired.
\end{proof}

Passing to the opposite category, we see that the category of semi-affinoid $K$-spaces has fibered products.
\subsubsection{The Nullstellensatz}

\begin{prop}\label{hilbertprop}
Semi-affinoid $K$-algebras are Jacobson rings.
\end{prop}
\begin{proof}
Any quotient of a semi-affinoid $K$-algebra is again semi-affinoid; hence it suffices to show that if $A$ is a semi-affinoid $K$-algebra and if $f\in A$ is a semi-affinoid function such that $f(x)=0$ for all $x\in\sSp A$, then $f$ is nilpotent. We may divide $A$ by its nilradical and thereby assume that $A$ is reduced. Let $\ul{A}$ be an $R$-model of ff type for $A$, and let $X=(\Spf\ul{A})^\rig$ denote the rigid-analytic generic fiber of $\Spf \ul{A}$. Since $A$ is excellent, being a localization of the excellent ring $\ul{A}$, and since rigid $K$-spaces are excellent (cf.\ \cite{conrad_irred} 1.1), it follows from \cite{dejong_crystalline} Lemma 7.1.9 that the space $X$ is reduced and that we may view $A$ as a subring of $\Gamma(X,\O_X)$ such that the value of $f$ in a point $x\in X$ agrees with the value of $f$ in the corresponding maximal ideal of $A$. Since $f(x)=0$ for all $x\in X$, we see that $f=0$ as a function on $X$ and, hence, in $A$.
\end{proof}

\subsection{Semi-affinoid spaces}

\subsubsection{The rigid space associated to a semi-affinoid $K$-space}\label{rigspaceofsemaffspacesec}

Let $X=\sSp A$ be a semi-affinoid $K$-space. An affine flat formal model of ff type for $X$ is an affine flat formal $R$-scheme of ff type $\fX$ together with an identification of $\Gamma(\fX,\O_\fX)$ with an $R$-model of ff type for $A$. By Definition \ref{saffalgdefi}, every semi-affinoid $K$-space admits an affine flat model of ff type. There is an obvious generic fiber functor $\urig$ from the category of affine flat formal $R$-schemes of ff type to the category of semi-affinoid $K$-spaces, given by
\[
(\Spf\ul{A})^\urig\,:=\,\sSpns(\ul{A}\otimes_RK)\;.
\]
Let $\fX$ be a flat affine $R$-model of ff type for $X$. Berthelot's construction yields a rigid $K$-space $X^\r\mathrel{\mathop:}=\fX^\rig$ together with a $K$-homomorphism 
\[
\phi\colon A\rightarrow\Gamma(X^\r,\O_{X^\r})\;,
\]
cf.\ \cite{dejong_crystalline} 7.1.8. By our discussion in Section \ref{specmapsec} and by \cite{dejong_crystalline} 7.1.9, the homomorphism $\phi$ induces a bijection $|X^\r|\rightarrow|X|$ and local homomorphisms $A_\m\rightarrow\O_{X^\r,x}$ which are isomorphisms on maximal-adic completions, where $x$ is a point of $X^\r$ and where $\m\in\Max A$ is the image of $x$ under the above bijection. We say that $X^\r$ is the \emph{rigid space associated to} $X$ via Berthelot's construction.\index{rigid space!of a semi-affinoid space} It is independent of the choice of $\fX$, the pair $(X^\r,\phi)$ being characterized by the following universal property:

\begin{prop}\label{semaffspaceassocprop}
Let $Y$ be a rigid $K$-space, and let $\psi\colon A\rightarrow\Gamma(Y,\O_Y)$ be a $K$-algebra homomorphism. There exists a unique morphism of rigid $K$-spaces $\sigma\colon Y\rightarrow X^\r$ such that $\psi=\Gamma(\sigma^\sharp)\circ\phi$. 
\end{prop}
\begin{proof}
Uniqueness of $\sigma$ follows from the above-mentioned fact that $\phi$ induces a bijection of points and isomorphisms of completed stalks; we may thus assume that $Y$ is affinoid, $Y=\Sp B$.  Let $\ul{A}\subseteq A$ be the $R$-model of ff type corresponding to $\fX$. By Corollary \ref{freesemaffcor} ($iv$) and Corollary \ref{afflattfintypecor}, $\psi$ restricts to an $R$-homomorphism $\ul{\psi}\colon\ul{A}\rightarrow\ul{B}$, where $\ul{B}$ is a suitable $R$-model of tf type for $B$; now $\sigma\mathrel{\mathop:}=(\Spf\ul{\psi})^\rig$ has the required properties.
\end{proof}

If $\ul{\tau}\colon\fY\rightarrow\fX$ is a morphism of affine flat formal $R$-schemes of ff type and if $\ul{\tau}^\urig$ denotes the induced morphism of associated semi-affinoid $K$-spaces, we easily see that the unique morphism $(\ul{\tau}^\srig)^r$ provided by Proposition \ref{semaffspaceassocprop} is given by $\ul{\tau}^\rig$.

\subsubsection{Semi-affinoid subdomains}

Closed subspaces of semi-affinoid $K$-spa\-ces are easily defined in the usual way:

\begin{defi}\label{semaffclosedimdefi}
A morphism of semi-affinoid $K$-spaces is called a \emph{closed immersion} if it corresponds to a surjective homomorphism of semi-affinoid $K$-algebras. A \emph{closed semi-affinoid subspace} of a semi-affinoid $K$-space is an equivalence class of closed immersions, where two closed immersions of uniformly rigid $K$-spaces $i_1:Y_1\rightarrow X$ and $i_2:Y_2\rightarrow X$ are called equivalent if there exists an isomorphism $\phi:Y_1\overset{\sim}{\rightarrow}Y_2$ such that $i_1=i_2\circ\phi$.
\end{defi}

If $A$ is a semi-affinoid $K$-algebra and if $I\subseteq A$ is an ideal, then the natural closed immersion $\sSp A/I\rightarrow\sSp A$ is clearly injective onto the set of maximal ideals containing $I$. Moreover, if $A\rightarrow C$ is a homomorphism of semi-affinoid $K$-algebras, then $A/I\hat{\otimes}_AC=C/IC$, because this quotient already represents the amalgamated sum of $C$ and $A/I$ over $A$ in the category of all $K$-algebras. In particular, closed immersions of semi-affinoid $K$-spaces are stable under the formation of fibered products.

To define a reasonable structure of G-topological $K$-space on the set of physical points of a semi-affinoid $K$-space $X$, it is natural to consider subsets $U$ of $X$ that canonically inherit a structure of semi-affinoid $K$-space:

\begin{defi}\label{semaffpresubdomdefi}
A subset $U$ in a semi-affinoid $K$-space $X$ is called \emph{re\-presentable} if there exists a morphism of semi-affinoid $K$-spaces to $X$ whose image lies in $U$ and which is final with this property. Such a morphism is said to represent all semi-affinoid morphisms to $X$ with image in $U$.
\end{defi}

\begin{remark}
Here we differ from the terminology used in the author's PhD thesis; there the representable subsets are called \emph{semi-affinoid pre-subdomains}, cf.\ \cite{mythesis} Section 1.3.3.
\end{remark}

Clearly $X$ and $\emptyset$ are representable subsets of $X$. Copying the proof of \cite{bgr} 7.2.2/1, we see that a morphism representing a subset $U\subseteq X$ is injective with image $U$ and that it induces isomorphisms of infinitesimal neighborhoods of points. Using the existence of fibered products in the category of semi-affinoid $K$-spaces, we see that representable subsets are preserved under pullback with respect to morphisms of semi-affinoid spaces. The universal property of representable subsets yields a presheaf $\O_X$ in semi-affinoid $K$-algebras on the category of representable subsets in $X$. 

In the category of affinoid $K$-spaces, the representable subsets are called affinoid subdomains (cf.\ \cite{bgr} 7.2.2/2), and they play a predominant role in the foundations of rigid geometry. In the uniformly rigid setting, we are unable to handle general representable subsets; for instance, we do not know whether representable subsets induce admissible open subsets via the functor $\r$ which is induced by Berthelot's construction. We will thus only consider representable subsets of a specific kind, which we call semi-affinoid subdomains:

\begin{defi}\label{semaffsubdomdefi}
A subset $U$ of a semi-affinoid $K$-space $X$ is called a semi-affinoid subdomain if there is an affine flat $R$-model of ff type $\fX$ for $X$ and a finite composition of open immersions, completion morphisms and admissible blowups $\phi:\fU\rightarrow\fX$ such that $\fU$ is affine and such that $U$ is equal to the image of $\phi^\urig$. We say that $\phi$ represents $U$ as a semi-affinoid subdomain in $X$. We say that $U$ is an elementary semi-affinoid subdomain in $X$ if $\phi$ can be chosen as an open immersion into an admissible blowup, and we say that $U$ is a retrocompact semi-affinoid subdomain in $X$ if $\phi$ can be chosen as a composition of open immersions and admissible blowups; such a $\phi$ is said to represent $U$ as an elementary or as a retrocompact semi-affinoid subdomain in $X$ respectively.
\end{defi}

In Corollary \ref{subdomunivcor}, we will see that semi-affinoid subdomains are actually representable in the sense of Definition \ref{semaffpresubdomdefi}. 

Open immersions of formal $R$-schemes of ff type induce retrocompact open immersions of rigid generic fibers, cf.\ \cite{dejong_crystalline} 7.2.2 and 7.2.4 (d). Moreover, completion morphisms induce (possibly non-retrocompact) open immersions of rigid generic fibers, cf.\ \cite{dejong_crystalline} 7.2.5, and admissible blowups induce isomorphisms of rigid generic fibers, cf.\ \cite{nicaise_traceformula} 2.19. Hence a semi-affinoid subdomain $U\subseteq X$ is admissibly open in $X^\r$. In particular, the $K$-homomorphism $\phi^{\urig,*}$ corresponding to $\phi^\urig$ is flat, since flatness is seen on the level of completions of stalks. Semi-affinoid subdomains may be regarded as nested rational subdomains defined in terms of strict or non-strict inequalities involving semi-affinoid functions. For example, the blowup of $\fX=\Spf R[[S]]$ in the ideal $(\pi, S)$ is covered by the affine open formal subschemes $\fX_1=\Spf (R[[S]]\langle V\rangle/(\pi V-S))\cong\Spf R\langle V\rangle$ and $\fX_2=\Spf (R[[S]]\langle W\rangle/(SW-\pi))$; the completion of $\fX_1$ along the ideal $(\pi, V)$ represents the open disc with radius $|\pi|$ within the open unit disc, while the completion of $\fX_2$ along $(\pi,W)$ defines the open annulus $|\pi|<|S|<1$.

\begin{remark}\label{nestedrequrem}
It is necessary to consider iterations as in Definition \ref{semaffsubdomdefi} because if $\fU$ is an open subset of a flat formal $R$-scheme of ff type $\fX$, then an admissible blowup of $\fU$ needs not extend to an admissible blowup of $\fX$, cf.\ \cite{mythesis} Example 1.1.3.12.
\end{remark}

In order to understand semi-affinoid subdomains, it will be useful to interpret strict transforms with respect to admissible blowups as pullbacks:

\begin{lem}\label{stricttrafocartlem}
Let $\fY\rightarrow\fX$ be a morphism of flat formal $R$-schemes of locally ff type, let $\fX'\rightarrow\fX$ be an admissible blowup, and let $\fY'\rightarrow\fY$ denote the induced admissible blowup of $\fY$, that is, the strict transform of $\fY$. Then the resulting square
\[
\begin{diagram}
\fY'&\rTo&\fX'\\
\dTo&&\dTo\\
\fY&\rTo&\fX
\end{diagram}
\]
is cartesian in the category of flat formal $R$-schemes of locally ff type.
\end{lem}
\begin{proof}
The universal property of the fibered product in the category of flat formal $R$-schemes of locally ff type is readily verified using the universal property of admissible blowups, the fact that the functor $\rig$ maps admissible blowups to isomorphisms and the fact that $\rig$ is faithful on the category of flat formal $R$-schemes of locally ff type.
\end{proof}

In the following, we write $\times'$ to denote the fibered product in the category of flat formal $R$-schemes of locally ff type. It is obtained from the usual fibered product by dividing out the coherent ideal of $\pi$-torsion; in particular, fibered products of affine flat formal $R$-schemes of ff type in the category of flat formal $R$-schemes of locally ff type are again affine. 

As we have just observed, admissible blowups of flat formal $R$-schemes are preserved under pullback in the category of flat formal $R$-schemes of locally ff type. The same is true for open immersions and completion morphisms, since they are flat and since they are preserved under pullback in the category of all formal $R$-schemes of locally ff type.

\begin{cor}\label{subdomunivcor}
Let $X$ be a semi-affinoid $K$-space, let $U\subseteq X$ be a semi-affinoid subdomain, and let $Y\rightarrow X$ be a morphism of semi-affinoid $K$-spaces.
\begin{enumerate}
\item The preimage of $U$ in $Y$ is a semi-affinoid subdomain in $Y$.
\item If $\fU\rightarrow\fX$ represents $U$ as a semi-affinoid subdomain in $X$ and if $\fY\rightarrow\fX$ is a model of $Y\rightarrow X$, then the projection $\fU\times'_\fX\fY\rightarrow\fY$ represents the preimage of $U$ as a semi-affinoid subdomain in $Y$.
\item If $\phi$ re\-pre\-sents $U$ as a semi-affinoid subdomain in $X$, then $\phi^\urig$ re\-pre\-sents all semi-affinoid morphisms to $X$ with image in $U$. In particular, semi-affinoid subdomains are representable in the sense of Definition \ref{semaffpresubdomdefi}.
\end{enumerate}
The analogous statements hold if we consider retrocompact or elementary semi-affinoid subdomains and their retrocompact or elementary representations.
\end{cor}
\begin{proof}
Statement ($ii$) implies statement ($i$) in view of Corollary \ref{freesemaffcor} ($iv$). To show ($ii$), let us consider a factorization 
\[
\fU\overset{\phi_n}{\longrightarrow}\fU_{n}\overset{\phi_{n-1}}{\longrightarrow}\cdots\overset{\phi_1}{\longrightarrow}\fU_1\overset{\phi_0}{\longrightarrow}\fX\quad(*)
\]
of $\fU\rightarrow\fX$, where the $\phi_i$ are admissible blowups, open immersions or completion morphisms. By the remarks preceding this Corollary, we see that the projection $\fU\times'_\fX\fY\rightarrow\fY$ defines a semi-affinoid subdomain in $Y$. Passing to associated rigid spaces, we see that this semi-affinoid subdomain coincides with the preimage of $U$ in $Y$. To prove ($iii$), let us write $\phi$ to denote $\fU\rightarrow\fX$, and let us assume that the image of $Y\rightarrow X$ lies in $U$; we must show that $Y\rightarrow X$ factors uniquely through $\phi^\urig$. Since $\phi^\urig$ induces an injection of physical points and isomorphisms of completed stalks, uniqueness  follows from Krull's Intersection Theorem. Let us show that the desired factorization exists. Again, Corollary \ref{freesemaffcor} ($iv$) shows that $Y\rightarrow X$ admits a model $\fY\rightarrow\fX$ with target $\fX$. Let us consider the pullback
\[
\fY_{n+1}\overset{\psi_n}{\longrightarrow}\fY_{n}\overset{\psi_{n-1}}{\longrightarrow}\cdots\overset{\psi_1}{\longrightarrow}\fY_1\overset{\phi_0}{\longrightarrow}\fY
\]
of $(*)$ under $\fY\rightarrow\fX$ in the category of flat formal $R$-schemes of locally ff type; then $\fY_{n+1}$ is affine, and all $\psi_i$ that are open immersions or completion morphisms are isomorphisms: Indeed, $Y\rightarrow X$ factors through $U$, specialization maps are surjective onto the sets of closed points of flat formal $R$-schemes of locally ff type, and  the closed points lie very dense in formal $R$-schemes of this type. Hence, the composition $\fY_{n+1}\rightarrow\fY$ is a composition of admissible blowups; by \cite{temkin_desing} 2.1.6, it is an admissible blowup. Since $\fY_{n+1}$ is affine, \cite{egaiii} 3.4.2 shows that $\fY_{n+1}\rightarrow\fY$ is a finite admissible blowup. After applying $\urig$, we thus obtain the desired factorization of $Y\rightarrow X$.
\end{proof}

By Corollary \ref{subdomunivcor} ($iii$), every semi-affinoid subdomain may be viewed as a semi-affinoid $K$-space in a natural way. 

\begin{question}\label{reprimpliessubdomrem}
One may ask whether every representable subset of a semi-affinoid $K$-space is in fact a semi-affinoid subdomain. Unfortunately, we do not know the answer.
\end{question}

\begin{cor}\label{semaffsubdomcor}
Let $X$ be a semi-affinoid $K$-space, let $U\subseteq X$ be a semi-affinoid subdomain, and let $\fX$ be a flat affine $R$-model of ff type for $X$. Then there exists a representation of $U$ as a semi-affinoid subdomain in $X$ with target $\fX$.
\end{cor}
\begin{proof}
Let $\fU'\rightarrow\fX'$ be a representation of $U$ as a semi-affinoid subdomain in $X$, let us write $X=\sSp A$, and let $\ul{A},\ul{A}'\subseteq A$ be the $R$-models of ff type of $A$ corresponding to $\fX$ and $\fX'$ respectively. By Corollary \ref{freesemaffcor} ($iv$) applied to the identity on $A$, there exists an $R$-model of ff type $\ul{A}''$ of $A$ containing both $\ul{A}$ and $\ul{A}'$. By Corollary \ref{freesemaffcor} ($ii$), the inclusions $\ul{A}\subseteq\ul{A}''$ and $\ul{A}'\subseteq\ul{A}''$ correspond to finite admissible blowups $\fX''\rightarrow\fX$ and $\fX''\rightarrow\fX'$. By Corollary \ref{subdomunivcor} ($ii$), the strict transform $\fU''\rightarrow\fX''$ of $\fU'\rightarrow\fX'$ under $\fX''\rightarrow\fX'$ represents $U$ as a semi-affinoid subdomain in $X$. Composing this representation with the admissible blowup $\fX''\rightarrow\fX$, we obtain a representation $\fU''\rightarrow\fX$ of $U$ as a semi-affinoid subdomain in $X$ with target $\fX$, as desired.
\end{proof}

\begin{remark}
One can easily show that if $U\subseteq X$ is a semi-affinoid subdomain and if $\fY\rightarrow\fX$ is a model of the inclusion of $U$ into $X$, then there exists a finite admissible blowup $\fY'$ of $\fY$ such that the composition $\fY'\rightarrow\fX$ represents $U$ as a semi-affinoid subdomain in $X$; this fact will not be needed in the following. 
\end{remark}

\begin{cor}\label{semaffsubdomtranscor}
Let $X$ be a semi-affinoid $K$-space.
\begin{enumerate}
\item Let $U\subseteq X$ be a semi-affinoid subdomain, and let $V$ be a subset of $U$. Then $V$ is a semi-affinoid subdomain in $U$ if and only if it is a semi-affinoid subdomain in $X$.
\item The set of semi-affinoid subdomain in $X$ is stable under the formation of finite intersections.
 \end{enumerate}
\end{cor}
\begin{proof}
If $V$ is semi-affinoid in $X$, then $V=V\cap U$ is semi-affinoid in $U$ by  Corollary \ref{subdomunivcor} ($i$). Conversely, assume that $V$ is semi-affinoid in $U$, and let $\fU\rightarrow\fX$ be a representation of $U$ as a semi-affinoid subdomain in $X$. By Corollary \ref{semaffsubdomcor}, there exists a representation $\fV\rightarrow\fU$ of $V$ as a semi-affinoid subdomain in $U$; the composition $\fV\rightarrow\fU\rightarrow\fX$ represents $V$ as a semi-affinoid subdomain in $X$. This settles the first statement. To show $(ii)$, let us consider two semi-affinoid subdomains $U$ and $V$ in $X$. By Corollary \ref{subdomunivcor} ($i$), $U\cap V$ is a semi-affinoid subdomain in $U$; by part ($i$), $U\cap V$ is thus a semi-affinoid subdomain in $X$.
\end{proof}

These results obviously remain true if we only consider retrocompact semi-affinoid subdomains instead of general semi-affinoid subdomains. Similarly, elementary semi-affinoid subdomains are preserved under pullback with respect to morphisms of semi-affinoid spaces. However, if $U$ is an elementary semi-affinoid subdomain in a semi-affinoid $K$-space $X$ and if $V$  is an an elementary semi-affinoid subdomain in $U$, then $V$ needs not be elementary in $X$. Likewise, if $U$ is a semi-affinoid subdomain in $X$ and if $V$ is a retrocompact semi-affinoid subdomain in $U$, then $V$ needs not be retrocompact in $X$.

We conclude this section by identifying retrocompact semi-affinoid subdomains in affinoid $K$-spaces:

\begin{lem}\label{veryspecialinafflem}
Let $A$ be an affinoid $K$-algebra; then a retrocompact semi-affinoid subdomain $U$ in $\sSp A$ is an affinoid subdomain in $\Sp A$.
\end{lem}
\begin{proof}
Let $\phi\colon\fY\rightarrow \fX$ be a morphism defining $U$ as a retrocompact semi-affinoid subdomain in $X$. By Corollary \ref{afflattfintypecor}, $\fX$ is of tf type over $R$. Since $\phi$ is adic, $\fY$ is of tf type over $R$ as well, such that $\phi^\rig$ is a morphism of affinoid $K$-spaces. By  Corollary \ref{subdomunivcor} ($iii$), $\phi$ represents all semi-affinoid maps with image in $U$; in particular it represents all affinoid maps with image in $U$. Hence, $U$ is an affinoid subdomain in $\Sp A$.
\end{proof}

Conversely, it is clear that for any affinoid $K$-algebra $A$, the rational subdomains in $\Sp A$ define semi-affinoid subdomains in $\sSp A$. Let $U\subseteq\Sp A$ be a general affinoid subdomain in $\Sp A$. By the Theorem of Gerritzen and Grauert (\cite{bgr} 7.3.5/1), $U$ is a finite union of rational subdomains. Let $\fX$ be any affine flat formal $R$-model of tf type for $\Sp A$. By \cite{frg1} Lemma 4.4, there exist an admissible formal blowup $\fX'\rightarrow\fX$ of $\fX$ and an open formal subscheme $\fU\subseteq\fX'$ such that $U=\fU^\rig$. However, we do not know whether $\fU$ is affine, so we do not know whether a general affinoid subdomain $U$ in $\Sp A$ is a semi-affinoid subdomain or even a representable subset in $\sSp A$. Nonetheless, we will see that affinoid subdomains in $\Sp A$ are admissible open in the uniformly rigid G-topology on $\sSp A$, cf.\ Proposition \ref{retroadmprop}.

\subsubsection{G-topologies on semi-affinoid spaces}

We first define an auxiliary G-to\-po\-lo\-gy $\sT_\aux$ on the category of semi-affinoid $K$-spaces equipped with the physical points functor, cf.\ \cite{bgr} 9.1.2. The $\sT_\aux$-admissible subsets of a semi-affinoid $K$-space are the semi-affinoid subdomains of that space. If $I$ is a rooted tree and if $i\in I$ is a vertex, we let $\children(i)$ denote the set of children of $i$.
\begin{defi}\label{maintauxdefi}
Let $X$ be a semi-affinoid $K$-space, and let $(X_i)_{i\in I}$ be a finite family of semi-affinoid subdomains in $X$.
\begin{enumerate}
\item We say that $(X_i)_{i\in I}$ is an elementary covering of $X$ if there exist an affine flat $R$-model of ff type $\fX$ for $X$, an admissible blowup $\fX'\rightarrow\fX$ and an affine open covering $(\fX_i)_{i\in I}$ of $\fX'$ such that for each $i\in I$, $\fX_i\subseteq\fX'\rightarrow\fX$ represents $X_i$ as a semi-affinoid subdomain in $X$.
\item We say that $(X_i)_{i\in I}$ is a treelike covering of $X$ if there exists a rooted tree structure on $I$ such that $X_r=X$, where $r$ is the root of $I$, and such that $(X_j)_{j\in\children(i)}$ is an elementary covering of $X_i$ for all $i\in I$ which are not leaves. A rooted tree structure on $I$ with these properties is called suitable for $(X_i)_{i\in I}$.
\item We say that $(X_i)_{i\in I}$ is a leaflike covering if it extends to a treelike covering $(X_i)_{i\in J}$, $J\supseteq I$, where $J$ admits a suitable rooted tree structure such that $I$ is identified with the set of leaves of $J$.
\item We say that $(X_i)_{i\in I}$ is $\sT_\aux$-admissible if it admits a leaflike refinement.
\end{enumerate}
\end{defi}

If $(X_i)_{i\in I}$ is an elementary, treelike or leaflike covering of $X$, then by definition all $X_i$ are retrocompact in $X$. For trivial reasons, condition ($iv$) in Definition \ref{maintauxdefi} can  be checked after refinement.

Arguing as in the proof of Corollary \ref{semaffsubdomcor}, we see that an elementary covering can be represented with respect to any flat affine $R$-model of ff type $\fX$ of $X$. It follows that any treelike covering $(X_i)_{i\in I}$ together with a suitable rooted tree structure on $I$ admits a model; that is, we have
\begin{enumerate}
\item for each $i\in I$, an affine flat $R$-model of ff type $\fX_i$ for $X_i$,
\item for each inner $i\in I$, an admissible blowup $\fX_i'\rightarrow\fX_i$ and
\item for each inner $i\in I$ and for each child $j$ of $i$, an open immersion  $\fX_j\hookrightarrow\fX_i'$
such that $\fX_j\subseteq\fX_i'\rightarrow\fX_i$ represents $X_j$ as a semi-affinoid subdomain in $X_i$.
\end{enumerate}

Arguing as in the proof of Corollary \ref{subdomunivcor}, we see that elementary, treelike and leaflike coverings, suitable rooted tree structures and models in the above sense are preserved under pullback with respect to morphisms $Y\rightarrow X$ of semi-affinoid $K$-spaces and their models $\fY\rightarrow\fX$, where $\fY$ and $\fX$ are flat affine models of ff type for $Y$ and $X$ respectively.

\begin{lem}\label{tauxtranslem}
Let $X$ be a semi-affinoid $K$-space, let $(U_i)_{i\in I}$ be a covering of $X$ by semi-affinoid subdomains, and for each $i\in I$, let $(V_{ij})_{j\in J_i}$ be a covering of $U_i$. If all of these coverings are leaflike or $\sT_\aux$-admissible, then the same holds for the covering $(V_{ij})_{i\in I,j\in J_i}$ of $X$.
\end{lem}
\begin{proof}
Let us first consider the case where the given coverings are leaflike.
 Let us choose a treelike covering $(U_i)_{i\in I'}$  of $U$ extending $(U_i)_{i\in I}$ together with a suitable rooted tree structure on $I'$ such that $I\subseteq I'$ is the set of leaves. Similarly, for each $i\in I$ we choose a treelike covering $(V_{ij})_{j\in J_i'}$ extending $(V_{ij})_{j\in J_i}$ together with a suitable rooted tree structure on $J_i'$ such that $J_i\subseteq J_i'$ is identified with the set of leaves for all $i\in I$. For each $i\in I$, we glue the rooted tree $J_i'$ to the rooted tree $I'$ by identifying the root of $J_i'$ with the leaf $i$ of $I'$. We obtain a rooted tree $J'$ whose set of leaves is identified with the disjoint union of the sets $J_i$, $i\in I$. For each $i\in I$, $U_i=V_{i r_i}$, where $r_i$ is the root of $J_i'$; hence we obtain a covering $(V_j)_{j\in J'}$ such that the given rooted tree structure on $J'$ is suitable for $(V_j)_{j\in J'}$; indeed, this can be checked locally on the rooted tree $J'$. We conclude that the composite covering $(V_{ij})_{i\in I,j\in J_i}$ of $X$ is leaflike. The statement for $\sT_\aux$-admissible coverings now follows by passing to leaflike refinements.
\end{proof}

Combining Lemma \ref{tauxtranslem} and the fact that $\sT_\aux$-admissible coverings are stable under pullback, we see that the semi-affinoid subdomains and the $\sT_\aux$-admissible coverings define a G-topology on the category of semi-affinoid $K$-spaces equipped with the physical points functor. The following proposition suggests that $\sT_\aux$ should be viewed as an analog of the weak G-topology in rigid geometry. We first define:

\begin{defi}
A retrocompact covering of a semi-affinoid $K$-space $X$ is a finite family of retrocompact semi-affinoid subdomains of $X$ that covers $X$ on the level of physical points.
\end{defi}

If $I$ is a rooted tree, we write $\leaves(I)$ to denote the set of leaves of that tree, and we write $v(I)$ denote the volume of the tree, that is, its number of vertices.

\begin{prop}\label{retroadmprop}
Retrocompact coverings of semi-affinoid spaces are $\sT_\aux$-admissible.
\end{prop}
\begin{proof}
Let $X$ be a semi-affinoid $K$-space, and let $(X_i)_{i\in I}$ be a finite family of retrocompact semi-affinoid subdomains in $X$ covering $X$ on the level of sets; we have to show that $(X_i)_{i\in I}$ is $\sT_\aux$-admissible. For each $i\in I$, we choose a retrocompact representation $\phi_i$ of $X_i$ in $X$, such that the targets of the $\phi_i$ all coincide with a fixed flat affine target $\fX$. For each $i\in I$, we choose a factorization
\[
\phi_i\,=\,\beta_{i1}\circ\psi_{i1}\circ\cdots\circ\beta_{in_i}\circ\psi_{in_i}\;,
\]
where the $\psi_{ij}$ are open immersions and the $\beta_{ij}$ are admissible blowups,
\[
\fX_{ij}\overset{\psi_{ij}}{\hookrightarrow}\fX'_{ij}\overset{\beta_{ij}}\rightarrow\fX_{i,j-1}\;,
\]
with $\fX_{i0}=\fX$. Let $v$ denote the sum of the $n_i$; we say that $v$ is the total length of the given retrocompact representation. Let $\fX'\rightarrow\fX$ be an admissible blowup dominating all $\beta_{i1}:\fX'_{i1}\rightarrow\fX$, and let $\fU_{i}\subseteq\fX'$ denote the preimage of $\fX_{i1}\subseteq\fX'_{i1}$. The $\fX_{i1}^\rig$ cover $\fX^\rig$, the specialization map $\sp_{\fX'}$ is surjective onto the closed points of $\fX'$, and the closed points in $\fX'$ lie very dense; hence $\fX'$ is covered by the $\fU_i$. For each $i\in I$, we consider the pullback $\psi_i'$ of
\[
\beta_{i2}\circ\psi_{i2}\circ\cdots\circ\beta_{in_i}\circ\psi_{in_i}
\]
under $\fU_i\subseteq\fX'\rightarrow\fX_{i1}'$, and moreover for each $j\in I$ different from $i$ we consider the pullback $\phi_{ij}'$ of
\[
\phi_j\,=\,\beta_{j1}\circ\psi_{j1}\circ\cdots\circ\beta_{jn_j}\circ\psi_{jn_j}
\]
under $\fU_i\subseteq\fX'\rightarrow\fX$, both in the category of flat formal $R$-schemes of ff type. For each $i\in I$, we choose a finite affine covering of $\fU_i$. For each constituent $\fV_{is}$ of this covering with semi-affinoid generic fiber $V_{is}$, we choose finite affine coverings of $(\psi'_i)^{-1}(\fV_{is})$ and of $(\phi'_{ij})^{-1}(\fV_{is})$, for $j\in I\setminus\{i\}$. We obtain a retrocompact covering of $V_{is}$, together with retrocompact representations as above of total length $v-1$. If we let $i$ and $s$ vary, the resulting retrocompact covering of $X$ refines $(X_i)_{i\in I}$. Since the $V_{is}$, for varying $i$ and $s$, form an elementary covering of $X$, it suffices to see that the given retrocompact covering of $V_{is}$ is $\sT_\aux$-admissible, which now follows by induction on $v$; the case $v=1$ is trivial.
\end{proof}

\begin{defi}Let $\sT_\urig$ denote the finest G-topology on the category of semi-affinoid $K$-spaces which is slightly finer than $\sT_\aux$ in the sense of \cite{bgr} 9.1.2/1.
\end{defi}

The G-topology $\sT_\urig$ is called the uniformly rigid G-topology. It exists by \cite{bgr} 9.2.1/2, and it is saturated in the sense that it satisfies conditions (G$_0$)--(G$_2$) in \cite{bgr} 9.1.2, saying that $\sT_\urig$-admissibility of subsets can be checked locally with respect to $\sT_\urig$-admissible coverings and that admissibility of a covering by $\sT_\urig$-admissible subsets can be checked after refinement.

As a corollary of [BGR] 9.1.2/3, we obtain the following explicit description of the uniformly rigid G-topology on a semi-affinoid $K$-space:

\begin{prop}\label{explicitprop}
Let $X$ be a semi-affinoid $K$-space.
\begin{enumerate}
\item A subset $U\subseteq X$ is $\sT_\srig$-admissible if and only if it admits a covering $(U_i)_{i\in I}$ by semi-affinoid subdomains such that for any morphism $\phi\colon Y\rightarrow X$ of semi-affinoid $K$-spaces with $\phi(Y)\subseteq U$, the induced covering of $Y$ has a leaflike refinement.
\item A covering $(U_i)_{i\in I}$ of a $\sT_\srig$-admissible subset $U$ in $X$ by $\sT_\srig$-admissible subsets is $\sT_\srig$-admissible if and only if for any morphism $\phi\colon Y\rightarrow X$ of semi-affinoid $K$-spaces with $\phi(Y)\subseteq U$, the induced covering of $Y$ has a leaflike refinement.
\end{enumerate}
\end{prop}

\begin{cor}\label{srigpropcor}
Let $X$ be a semi-affinoid $K$-space. 
\begin{enumerate}
\item For any semi-affinoid subdomain $U$ of $X$, the uniformly rigid G-topology on $X$ restricts to the uniformly rigid G-topology on $U$.
\item If $U\subseteq X$ is a finite union of retrocompact semi-affinoid subdomains in $X$, then $U$ is $\sT_\urig$-admissible, and every finite covering of $U$ by retrocompact semi-affinoid subdomains in $X$ is $\sT_\urig$-admissible.
\end{enumerate}
\end{cor}
\begin{proof}
By Corollary \ref{semaffsubdomtranscor} ($i$), the semi-affinoid subdomains in $U$ are the semi-affinoid subdomains in $X$ contained in $U$, and by Corollary \ref{subdomunivcor} ($iii$) the semi-affinoid morphisms to $X$ with image in $U$ correspond to the semi-affinoid morphisms to $U$. Hence, statement ($i$) follows from Proposition \ref{explicitprop} ($i$) and ($ii$). 

To prove the second statement, let $(U_i)_{i\in I}$ be a finite family of retrocompact semi-affinoid subdomains of $X$ such that $U$ is the union of the $U_i$. Let $Y$ be any semi-affinoid $K$-space, and let $\phi\colon Y\rightarrow X$ be any semi-affinoid morphism whose image is lies in $U$. Then $(\phi^{-1}(U_i))_{i\in I}$ is a retrocompact covering of $Y$; by Propostion \ref{retroadmprop}, it admits a leaflike refinement. By Proposition \ref{explicitprop} ($i$), we conclude that $U$ is a $\sT_\srig$-admissible subset of $X$, and by Proposition \ref{explicitprop} ($ii$) we see that the covering $(U_i)_{i\in I}$ of $U$ is $\sT_\srig$-admissible.
\end{proof}

In particular, Corollary \ref{srigpropcor} ($ii$) and the theorem of Gerritzen and Grauert \cite{bgr} 7.3.5/1 show that if $A$ is an affinoid $K$-algebra and if $U\subseteq\Sp A$ is an affinoid subdomain, then $U\subseteq\sSp A$ is $\sT_\urig$-admissible.

\begin{remark}[quasi-compactness]\label{qcrem} Proposition \ref{explicitprop} ($ii$) shows that semi-affinoid $K$-spaces are quasi-compact in $\sT_\urig$, cf.\ \cite{bgr} p.\ 337. By the maximum principle for affinoid $K$-algebras, it follows that $\sSp (R[[S]]\otimes_RK)$ has no $\sT_\urig$-admissible covering by semi-affinoid subdomains whose rings of functions are affinoid. In particular, the covering of $\sSp (R[[S]]\otimes_RK)$ provided by Berthelot's construction is not $\sT_\urig$-admissible. 
\end{remark}

\begin{remark}[bases for $\sT_\urig$]
Proposition \ref{explicitprop} implies that the semi-affinoid subdomains form a basis for the uniformly rigid G-topology on a semi-affinoid $K$-space, cf.\ \cite{bgr} p.\ 337. The retrocompact semi-affinoid subdomains in $\sSp (K\langle S\rangle)$ do not form a basis for $\sT_\urig$: Indeed, $\sSp (R[[S]]\otimes_RK)$ is a semi-affinoid subdomain in $\sSp (K\langle S\rangle)$; by Lemma \ref{veryspecialinafflem} and Remark \ref{qcrem}, it does not admit a $\sT_\urig$-admissible covering by retrocompact semi-affinoid subdomains in $\sSp (K\langle S\rangle)$. Thus, even though the $K$-algebra $K\langle S\rangle$ is affinoid, the uniformly rigid G-topology on $\sSp (K\langle S\rangle)$ turns out to be strictly coarser than the rigid G-topology on $\Sp (K\langle S\rangle)$. We do not know whether this discrepancy already appears on the level of admissible subsets.
\end{remark}

We conclude our discussion of the uniformly rigid G-topology $\sT_\urig$ by showing that it is finer than the Zariski topology $\sT_\Zar$ which, on a semi-affinoid $K$-space $X$, is generated by the non-vanishing loci $D(f)$ of semi-affinoid functions $f$ on $X$:

\begin{prop}
The uniformly rigid G-topology $\sT_\urig$ is finer than the Zariski topology $\sT_\rig$.
\end{prop}
\begin{proof}
Let $X=\sSp A$ be a semi-affinoid $K$-space, let $U\subseteq X$ be a Zariski-open subset, and let $f_1,\ldots,f_n\in A$ be semi-affinoid functions such that $U$ is the union of the Zariski-open subsets $D(f_i)=\{x\in\Max A\,;\,f_i(x)\neq 0\}$. Let $Y$ be a nonempty semi-affinoid $K$-space, and let $\phi\colon Y\rightarrow X$ be a semi-affinoid morphism whose image is contained in $U$.
For each $i$, the preimage $\phi^{-1}(D(f_i))$ is the set of points $y\in Y$ where $\phi^*f_i\neq 0$. Since $Y$ is covered by the $\phi^{-1}(D(f_i))$, the $\phi^*f_i$ generate the unit ideal in $B$. That is, there exist elements $b_1,\ldots,b_n$ in $B$ such that $b_1\phi^*f_1+\ldots +b_n\phi^*f_n=1$. Let us set $\gamma\mathrel{\mathop:}=(\max_i|b_i|_\sup)^{-1}$; this number is well-defined since the $b_i$ are bounded functions on $Y$ without a common zero. By the strict triangle inequality, $\max_i|\phi^*f_i(y)|\geq\gamma$ for all $y\in Y$. For each $i$, let $Y_i\subseteq Y$ denote the set of points $y\in Y$ where $|\phi^*f_i(y)|\geq\gamma$; then $(Y_i)_{1\leq i\leq n}$ is a retrocompact covering of $Y$ refining $(\phi^{-1}(D(f_i)))_{1\leq i\leq n}$. By Proposition \ref{retroadmprop}, retrocompact coverings are $\sT_\srig$-admissible; hence $U\subseteq X$ is $\sT_\srig$-admissible. If $(U_j)_{j\in J}$ is a Zariski-covering of $U$, we may pass to a refinement and assume that for all $j\in J$, $U_j=D(g_j)\subseteq X$ for some semi-affinoid function $g_j$ on $X$; we can then argue along the same lines to prove that $(U_j)_{j\in J}$ is a $\sT_\srig$-admissible covering of $U$.
\end{proof}

The above argument works even though the maximum principle might fail on $Y$. Let us point out that our proof shows the following: If $f_1,\ldots,f_n$ are semi-affinoid functions on $X$, if 
\[
U\,=\,\bigcup_{i=1}^nD(f_i)
\]
is the associated Zariski-open subset of $X$, and if we set 
\[
U_{\geq\varepsilon}\,=\,\bigcup_{i=1}^n\,\{x\in X\,;\,|f_i(x)|\geq\varepsilon\}
\]
for $\varepsilon\in\sqrt{|K^*|}$, then the resulting covering $(U_{\geq\varepsilon})_{\varepsilon}$ of $U$ by finite unions of retrocompact semi-affinoid subdomains of $X$ is $\sT_\srig$-admissible. In particular, Zariski-open subsets in semi-affinoid spaces need not be quasi-compact in the uniformly rigid G-topology. As a consequence, the sheaf of uniformly rigid functions on a semi-affinoid $K$-space, to be defined in the following section, may have unbounded sections on Zariski-open subsets.


\subsubsection{The acyclicity theorem}
Let $X$ be a semi-affinoid $K$-space. We show that the presheaf $\O_X$ that we introduced after Definition \ref{semaffpresubdomdefi} is a sheaf for $\sT_\aux$ and, hence, extends uniquely to a sheaf for $\sT_\urig$. More generally, we show that every $\O_X$-module associated to a finite module over the ring of global functions on $X$ is \emph{acyclic} for any $\sT_\urig$-admissible covering of $X$ in the sense of \cite{bgr} p.\ 324. Adopting methods from \cite{luetkebohmertformalrigid}, we derive our acyclicity theorem from results in formal geometry; we also use ideas from \cite{lipshitz_robinson} III.3.2.

Let us recall from \cite{bgr} p. 324 that if $\sF$ is a presheaf in $\O_X$-modules on $\sT_\aux$, a covering $(X_i)_{i\in I}$ of $X$ by semi-affinoid subdomains is called $\sF$-\emph{acyclic} if the associated augmented \v{C}ech complex is acyclic. The covering $(X_i)_{i\in I}$ is called \emph{universally} $\sF$-acyclic if $(X_i\cap U)_{i\in I}$ is $\sF|_U$-acyclic for any semi-affinoid subdomain $U\subseteq X$.

\begin{thm}\label{acyclicitytheorem}
For a semi-affinoid $K$-space $X$, $\sT_\aux$-admissible coverings are $\O_X$-acyclic. 
\end{thm}
\begin{proof}
Let us first consider an elementary covering $(X_i)_{i\in I}$. Let us choose a formal representation $(\fX,\beta\colon\fX'\rightarrow\fX,(\fX_i)_{i\in I})$ of $(X_i)_{i\in I}$, where $\beta$ is an admissible blowup and where $(\fX_i)_{i\in I}$ is a finite affine covering of $\fX'$ such that $\fX_i\subseteq\fX'\rightarrow\fX$ represents $X_i$ in $X$. By the ff type transcription of \cite{luetkebohmertformalrigid} 2.1, $\beta^\sharp\otimes_RK$ is an isomorphism; hence $\beta$ induces a natural identification of augmented \v{C}ech complexes
\[
C^\bullet_{\textup{aug}}((X_i)_{i\in I},\O_X)\cong C^\bullet_{\textup{aug}}((\fX_i)_{i\in I},\O_{\fX'}\otimes_RK)\;.
\]
We have to show that the complex on the right hand side is acyclic. Since $\O_{\fX'}\otimes_RK$ is a sheaf on $\fX'$, it suffices to show that
\[
\check{H}^q((\fX_i)_{i\in I},\O_{\fX'}\otimes_RK)\,=\,0
\]
for all $q\geq 1$. Since $I$ is finite, we have an identification
\[
\check{H}^q((\fX_i)_{i\in I},\O_{\fX'}\otimes_RK)\,=\,\check{H}^q((\fX_i)_{i\in I},\O_{\fX'})\otimes_RK\;.
\]
By the Comparison Theorem \cite{egaiii} 4.1.5, 4.1.7 and by the Vanishing Theorem \cite{egaiii} 1.3.1, the higher cohomology groups of a coherent sheaf on an affine noetherian formal scheme vanish. Since the $\fX_i$ are affine, Leray's theorem implies that
\[
\check{H}^q((\fX_i)_{i\in I},\O_{\fX'})\,=\,H^q(\fX',\O_{\fX'})\;.
\]
By \cite{egaiii} 1.4.11, $H^q(\fX',\O_{\fX'})=\Gamma(\fX,R^q\beta_*\O_{\fX'})$, and by the ff type transcription of \cite{luetkebohmertformalrigid} 2.1 this module is $\pi$-torsion. We have thus finished the proof in the case where $(X_i)_{i\in I}$ is an elementary covering.

Let us turn towards the general case. By definition, every $\sT_\aux$-admissible covering of $X$ has a leaflike refinement; by \cite{bgr} 8.1.4/3 it is enough to show that the leaflike coverings of $X$ are universally $\O_X$-acyclic. Since leaflike coverings are preserved with respect to pullback under morphisms of semi-affinoid $K$-spaces, it suffices to show that any leaflike covering $(X_i)_{i\in I}$ of $X$ is $\O_X$-acyclic. 

Let $(X_j)_{j\in J}$ be a treelike covering of $X$ extending $(X_i)_{i\in I}$, and let us choose a suitable rooted tree structure on $J$ such that $I\subseteq J$ is identified with the set of leaves of $J$. We argue by induction on the volume of $J$. If $J$ has only one vertex, the covering $(X_i)_{i\in I}$ is trivial and, hence, $\O_X$-acyclic. Let us assume that $J$ has more than one vertex. Let $\iota\in I$ be a leaf of $J$ such that the length $l(\iota)$ of the path from $\iota$ to the root is maximal in $\{l(i)\,;\,i\in I\}$. Let $\iota'\mathrel{\mathop:}=\parent(\iota)$ denote the parent of $\iota$. By maximality of $l(\iota)$, all siblings $i\in\children(\iota')$ of $\iota$ are leaves of $J$. Let $J'\mathrel{\mathop:}=J\setminus\children(\iota')$ be the rooted subtree of $J$ that is obtained by removing the siblings of $\iota$ (including $\iota$ itself). Then
\begin{packed_enum}
\item the set of leaves of $J'$ is $I'\mathrel{\mathop:}=(I\setminus\children(\iota'))\cup\{\iota'\}$,
\item $(X_j)_{j\in J'}$ is a treelike covering of $X$, and
\item $v(J')<v(J)$.
\end{packed_enum}
By our induction hypothesis, the covering $(X_i)_{i\in I'}$ is $\O_X$-acyclic. Since $(X_i)_{i\in I}$ is a refinement of $(X_i)_{i\in I'}$, \cite{bgr} 8.1.4/3 says that it suffices to prove that for any $r\geq 0$ and any tuple $(i_0,\ldots,i_r)\in(I')^{r+1}$, the covering $(X_i\cap X_{i_0\cdots i_r})_{i\in I}$ of $X_{i_0\cdots i_r}$ is $\O_X$-acyclic, where $X_{i_0\cdots i_r}$ denotes the intersection $X_{i_0}\cap\ldots\cap X_{i_r}$. Let us assume that there exists some $0\leq s\leq r$ such that $i_s\neq\iota'$. Then $i_s\in I$. Since $X_{i_0\cdots i_r}\subseteq X_{i_s}$, we see that the trivial covering of $X_{i_0\cdots i_r}$ refines $(X_i\cap X_{i_0\cdots i_r})_{i\in I}$. Since trivial coverings restrict to trivial coverings and since trivial coverings are acyclic, we deduce from \cite{bgr} 8.1.4/3 that $(X_i\cap X_{i_0\cdots i_r})_{i\in I}$ is acyclic. It remains to consider the case where all $i_s$, $0\leq s\leq r$, coincide with $\iota'$. That is, it remains to see that the covering $(X_i\cap X_{\iota'})_{i\in I}$ of $X_{\iota'}$ is $\O_X$-acyclic. It admits the elementary covering $(X_i)_{i\in\children(\iota')}$ as a refinement. Since elementary coverings restrict to elementary coverings and since elementary coverings are $\O_X$-acyclic by what we have shown so far, we conclude by \cite{bgr} 8.1.4/3 that $(X_i\cap X_{\iota'})_{i\in I}$ is $\O_X$-acyclic, as desired.
\end{proof}

By \cite{bgr} 9.2.3/1, $\O_X$ extends uniquely to a sheaf for $\sT_\urig$ which we again denote by $\O_X$ and which we call the structural sheaf or the sheaf of uniformly rigid functions. We can now easily discuss a fundamental example of a non-admissible finite covering of a semi-affinoid $K$-space by semi-affinoid subdomains:

\begin{example}\label{nonadmdisccovex}
The canonical covering of the semi-affinoid closed unit disc $\sSp (K\langle T\rangle)$ by the semi-affinoid open unit disc $\sSp (R[[T]]\otimes_RK)$ and the semi-affinoid unit circle $\sSp (K\langle T,T^{-1}\rangle)$ is not $\sT_\srig$-admissible and, hence, not $\sT_\textup{aux}$-admissible. Indeed, the two covering sets are nonempty and disjoint, while the ring of functions $K\langle T\rangle$ on the closed semi-affinoid unit disc has no nontrivial idempotents. 
\end{example}

If $X$ is a semi-affinoid $K$-space with ring of global functions $A$ and if $M$ is a finite $A$-module, the presheaf  $M\otimes\O_X$ sending a semi-affinoid subdomain $U$ in $X$ to $M\otimes_A\O_X(U)$ is an $\O_X$-module, which we call the $\O_X$-module \emph{associated} to $M$. A presheaf $\sF$ equipped with an $\O_X$-module structure is called \emph{associated}\index{module!associated} if it is isomorphic to $M\otimes\O_X$ for some finite $A$-module $M$. We sometimes abbreviate $\tilde{M}\mathrel{\mathop:}=M\otimes\O_X$.

\begin{cor}\label{moduleacythmcor}
Let $X$ be a semi-affinoid $K$-space, and let $\sF$ be an associated $\O_X$-module. Then every $\sT_\aux$-admissible covering $(X_i)_{i\in I}$ of $X$ is $\sF$-acyclic.
\end{cor}
\begin{proof}
By \cite{bgr} 8.1.4/3, we may assume that $I$ is finite. Using Theorem \ref{acyclicitytheorem}, the proof is now literally the same as the proof of \cite{bgr} 8.2.1/5.
\end{proof}

In particular, $M\otimes\O_X$ is a $\sT_\aux$-sheaf. By \cite{bgr} 9.2.3/1, $M\otimes\O_X$ extends uniquely to a $\sT_\srig$-sheaf on $X$ which we again denote by $M\otimes\O_X$ or by $\tilde{M}$ and which we call the sheaf associated to $M$.

\begin{remark}
If $U\subseteq X$ is a representable subset that is $\sT_\srig$-admissible, then $\O_X(U)=\O_U(U)$. Indeed, $U$ admits a $\sT_{\srig,X}$-admissible covering by semi-affinoid subdomains in $X$; since morphisms of semi-affinoid spaces are continuous for $\sT_\srig$, this covering is also $\sT_{\srig,U}$-admissible, so the statement follows from the fact that both $\O_X$ and $\O_U$ are $\sT_\srig$-sheaves. However, it is not clear whether $\sT_{\srig,X}$ restricts to $\sT_{\srig,U}$; for example, we do not know whether a semi-affinoid subdomain of $U$ is $\sT_{\srig,X}$-admissible. Of course, this does not affect our theory since we do not deal with general representable subsets.
\end{remark}

The category of abelian sheaves on $(X,\sT_\srig|_X)$ has enough injective objects, so the functor $\Gamma(X,\cdot)$ from the category of abelian sheaves on $X$ to the category of abelian groups has a right derived functor $H^\bullet(X,\cdot)$. By the Acyclicity Theorem and its Corollary \ref{moduleacythmcor}, this right derived functor can, for associated $\O_X$-modules, be calculated in terms of \v{C}ech cohomology:

\begin{cor}
Let $X$ be a semi-affinoid $K$-space, and let $\sF$ be an associated $\O_X$-module. Then the natural homomorphism $\check{H}^q(U,\sF)\rightarrow H^q(U,\sF)$
is an isomorphism for all $\sT_\srig$-admissible subsets $U\subseteq X$. In particular, $H^q(U,\sF)\,=\,0$ for all $q>0$ and all semi-affinoid subdomains $U\subseteq X$.
\end{cor}
\begin{proof}
The system $S$ of semi-affinoid subdomains in $X$ satisfies the following properties:
\begin{packed_enum}
\item $S$ is stable under the formation of intersections,
\item every $\sT_\srig$-admissible covering $(U_i)_{i\in I}$ of a $\sT_\srig$-admissible subset $U\subseteq X$ admits a $\sT_\srig$-admissible refinement by sets in $S$, and
\item  $\check{H}^q(U,\sF)$ vanishes for all $q>0$ and all $U\in S$;
\end{packed_enum}
hence the statement follows by means of the standard \v{C}ech spectral sequence argument.
\end{proof}

Transcribing the proof of 7.3.2/1, we see that if $A$ is a semi-affinoid $K$-algebra with associated semi-affinoid $K$-space $X$ and if $\m\subseteq A$ is a maximal ideal corresponding to a point $x\in X$, then the stalk $\O_{X,x}$ is local with maximal ideal $\m\O_{X,x}$ which coincides with the ideal of germs of functions vanishing in $x$. The arguments in the proof of \cite{bgr} 7.3.2/3 are also seen to work in our situation, showing that the natural homomorphisms $A/\m^{n+1}\rightarrow\O_{X,x}/\m^{n+1}\O_{X,x}$ are isomorphisms for all $n\in\N$. The rings $\O_{X,x}$ are noetherian, which can for example be seen by imitating the proof of \cite{bgr} 7.3.2/7. 

Transcribing the discussion at the beginning of \cite{bgr} 9.3.1, we see that the uniformly rigid G-topology and the sheaf of uniformly rigid functions define a functor from the category of semi-affinoid $K$-spaces to the category of locally ringed G-topological $K$-spaces. The proof of \cite{bgr} 9.3.1/2 carries over verbatim to the semi-affinoid situation, showing that this functor is fully faithful. We call a locally ringed G-topological $K$-space \emph{semi-affinoid} if it lies in the essential image of this functor. 

\subsection{Uniformly rigid spaces} We are now able to define the category of uniformly rigid $K$-spaces:
\begin{defi}\label{srigspacedefi}\index{uniformly rigid space}\index{category!of uniformly rigid $K$-spaces}
Let $X$ be a locally ringed G-topological $K$-space.
\begin{enumerate}
\item An admissible  \emph{semi-affinoid covering}\index{covering!semi-affinoid} of $X$ is an admissible covering $(X_i)_{i\in I}$ of $X$ such that for each $i\in I$, $(X_i,\O_X|_{X_i})$ is a semi-affinoid $K$-space.
\item The space $X$ is called \emph{uniformly rigid} if it satisfies conditions (G$_0$)--(G$_1$) in \cite{bgr} 9.1.2 and if it admits an admissible semi-affinoid covering.
\item An admissible open subset $U$ of a uniformly rigid $K$-space $X$ is called an \emph{open semi-affinoid subspace}\index{subspace!semi-affinoid} of $X$ if $(U,\O_X|_U)$ is a semi-affinoid $K$-space.
\end{enumerate}
\end{defi}

\begin{remark}
In the author's PhD thesis, open semi-affinoid subspaces were simply called semi-affinoid subspaces, cf.\ \cite{mythesis} Section 1.3.9.
\end{remark}

The category $\sRig_K$ of uniformly rigid $K$-spaces is a full subcategory of the category of locally G-topological $K$-spaces, and it contains the category of semi-affinoid $K$-spaces as a full subcategory.

\begin{remark}
We do not know whether an open semi-affinoid subspace $U$ of a semi-affinoid $K$-space $X$ is necessarily a semi-affinoid subdomain in $X$.  However, one easily verifies that $U$ is a representable subset in $X$. Moreover, one sees that $U$ is locally a semi-affinoid subdomain in $X$, cf.\ Lemma \ref{semaffsubspacecharlem} for a precise statement. In rigid geometry, the open affinoid subvarieties (cf.\ \cite{bgr} p.\ 357) of an affinoid space are precisely the affinoid subdomains, which means that there is no need to distinguish between the two notions in the affinoid setting.
\end{remark}

\begin{remark}
Let $X=\sSp A$ be a semi-affinoid $K$-space, and let $U=\sSp B$ be an open semi-affinoid subspace of $X$; then the restriction homomorphism $A\rightarrow B$ is flat. Indeed, for every maximal ideal $\n\subseteq B$ with corresponding point $x\in U$ and preimage $\m\subseteq A$, the induced homomorphism $A_\m\rightarrow B_\n$ induces an isomorphism of maximal-adic completions; by the Flatness Criterion \cite{bourbakica} III.5.2 Theorem 1, we conclude that $A\rightarrow B_\n$ is flat for all maximal ideals $\n$ in $B$, which implies that $A\rightarrow B$ is flat. 
\end{remark}






\begin{lem}\label{saffbasislem}
The open semi-affinoid subspaces of a uniformly rigid $K$-space $X$ form a \emph{basis} for the G-topology on $X$.
\end{lem}
\begin{proof}
Let $(X_i)_{i\in I}$ be an admissible semi-affinoid covering of $X$, and let $U\subseteq X$ be an admissible open subset. Then $(X_i\cap U)_{i\in I}$ is an admissible covering of $U$. For each $i\in I$, $X_i\cap U$ is admissible open in $X_i$ and, hence, admits an admissible covering by semi-affinoid subdomains of $X_i$. Hence, $U$ has an admissible semi-affinoid covering.
\end{proof}

It follows that if $X$ is a uniformly rigid $K$-space and if $U\subseteq X$ is an admissible open subset, then $(U,\O_X|_U)$ is a uniformly rigid $K$-space, again.

It is now clear that the Glueing Theorem \cite{bgr} 9.3.2/1 and its proof carry over verbatim to the uniformly rigid setting. Similarly, a morphism of uniformly rigid spaces can be defined locally on the domain; this is the uniformly rigid version of \cite{bgr} 9.3.3/1, and again the proof is obtained by literal transcription. Furthermore, a uniformly rigid $K$-space is determined by its functorial points with values in semi-affinoid $K$-spaces. 

We can also copy the proof of \cite{bgr} 9.3.3/2 to see that if $X$ is a semi-affinoid $K$-space and if $Y$ is a uniformly rigid $K$-space, then the set of morphisms from $Y$ to $X$ is naturally identified with the set of $K$-algebra homomorphisms from $\O_X(X)$ to $\O_Y(Y)$. 

Let $\fX$ be an affine formal $R$-scheme of ff type with semi-affinoid generic fiber $X$. The associated specialization map $\sp_\fX$ which we discussed in Section \ref{specmapsec} is naturally enhanced to a morphism of G-ringed $R$-spaces $\sp_\fX\colon X\rightarrow\fX$. Morphisms of uniformly rigid $K$-spaces being defined locally on the domain, we see that $\sp_\fX$ is \emph{final} among all morphisms of G-ringed $R$-spaces from uniformly rigid $K$-spaces to $\fX$. Using this universal property, we can invoke glueing techniques to construct the \emph{uniformly rigid generic} fiber $\fX^\srig$\index{generic fiber!uniformly rigid} of a general formal $R$-scheme of locally ff type $\fX$, together with a functorial specialization map $\sp_\fX\colon\fX^\srig\rightarrow\fX$ which is universal among all morphisms of G-ringed $R$-spaces from uniformly rigid $K$-spaces to $\fX$; this process does not involve Berthelot's construction. It is easily seen that $\urig$ is \emph{faithful} on the category of \emph{flat} formal $R$-schemes of locally ff type. A formal \emph{$R$-model}\index{model!of a uniformly rigid space} of a uniformly rigid $K$-space $X$ is a formal $R$-scheme $\fX$ of locally ff type together with an isomorphism $X\cong\fX^\srig$. The map $\sp_\fX$ is surjective onto the closed points of $\fX$ whenever $\fX$ is flat over $R$. This follows from Remark \ref{specmaprem}, together with the remark that the underlying topological space of $\fX$ is a Jacobson space, cf.\ \cite{egain} 0.2.8 and 6.4, so that the condition on a point in $\fX$ of being closed is local.

\begin{question}
Under what conditions does a uniformly rigid $K$-space admit a formal $R$-model?
\end{question}

By Proposition \ref{amalgsumsprop}, the category of semi-affinoid $K$-spaces has \emph{fibered products}\index{fibered product!of uniformly rigid spaces}; following the method outlined in \cite{bgr} 9.3.5, we see that the category of uniformly rigid $K$-spaces has fibered products as well and that these are constructed by glueing semi-affinoid fibered products of open semi-affinoid subspaces. It is clear from this description that the $\urig$-functor preserves fibered products.

Open semi-affinoid subspaces of semi-affinoid spaces can be described in the style of the Gerritzen-Grauert Theorem \cite{bgr} 7.3.5/3:

\begin{lem}\label{semaffsubspacecharlem}
Let $X$ be a semi-affinoid $K$-space, and let $U\subseteq X$ be an open semi-affinoid subspace. Then $U$ admits a leaflike covering $(U_i)_{i\in I}$ such that each $U_i$ is a semi-affinoid subdomain in $X$. 
\end{lem}
\begin{proof}
By Lemma \ref{saffbasislem}, $U$ admits an admissible covering $(V_j)_{j\in J}$ by semi-affinoid subdomains $V_j$ of $X$; by  Proposition \ref{explicitprop}, this covering is refined by a leaflike covering $(U_i)_{i\in I}$ of $U$. Via pullback, the $V_j$ are semi-affinoid subdomains of $U$. Let $\phi\colon I\rightarrow J$ denote a refinement map.  By Corollary \ref{semaffsubdomtranscor} ($i$), for each $i\in I$ the set $U_i$ is a semi-affinoid subdomain in $V_{\phi(i)}$ and, hence, in $X$, as desired. 
\end{proof}

A morphism of uniformly rigid $K$-spaces is called \emph{flat}\index{morphism!flat} in a point of its domain if it induces a flat homomorphism of stalks in this point, and it is called flat if it is flat in all points. Clearly a morphism of semi-affinoid $K$-spaces is flat in this sense if and only if the underlying homomorphism of rings of global sections is flat.
\subsubsection{Comparison with rigid geometry}\label{compfuncsec}

In Section \ref{rigspaceofsemaffspacesec}, we have defined the rigid space $X^\r$ associated to a semi-affinoid $K$-space $X=\sSp A$ together with a universal $K$-homomorphism $A\rightarrow\Gamma(X^\r,\O_{X^\r})$ which induces a bijection $X^\r\rightarrow X$ of physical points and isomorphisms of completed stalks. 
We will show that this universal homomorphism extends to a morphism $\comp_X:X^\r\rightarrow X$ of locally G-ringed $K$-spaces which is final among all morphisms from rigid $K$-spaces to $X$. To do so, we first show that the above bijection is continuous, that is, that the rigid G-topology $\sT_\rig$ is finer than $\sT_\urig$. We will need the following elementary fact from rigid geometry; the proof is left as an exercise to the reader:

\begin{lem}\label{easyrigidlem}
Let $X$ be an affinoid $K$-space, and let $U\subseteq X$ be a subset admitting a covering $(U_i)_{i\in I}$ by admissible open subsets $U_i\subseteq X$ such that for any affinoid $K$-space $Y$ and any morphism $\phi\colon Y\rightarrow X$ with image in $U$, the induced covering $(\phi^{-1}(U_i))_{i\in I}$ of $Y$ has a refinement which is a finite covering by affinoid subdomains. Then $U\subseteq X$ is admissible.
\end{lem}


\begin{prop}\label{srigrigcompprop}
The rigid G-topology $\sT_\rig$ on $X$ is finer than the uniformly rigid G-topology $\sT_\srig$. 
\end{prop}
\begin{proof}
It is clear that $\sT_\aux$-admissible subsets and $\sT_\aux$-admissible coverings are $\sT_\rig$-admissible. Let $U\subseteq X$ be a $\sT_\urig$-admissible subset. To check that $U$ is $\sT_\rig$-admissible, we  may work locally on $X^\r$. Let $V'\subseteq X^\r$ be an affinoid subspace; by Proposition \ref{semaffspaceassocprop}, the open immersion $V'\hookrightarrow X^\r$ corresponds to a morphism $V\rightarrow X$, where $V$ denotes the semi-affinoid $K$-space associated to $V'$ such that $V'=V^\r$. After pulling $U$ back under this morphism, we may thus assume that the $K$-algebra of global functions on $X$ is affinoid. Let $(U_i)_{i\in I}$ be a covering of $U$ by semi-affinoid subdomains in $X$ such that condition ($i$) of Proposition \ref{explicitprop} is satisfied. Let $Y$ be an affinoid $K$-space, and let $\phi\colon Y\rightarrow X^\r$ be a morphism of rigid spaces that factors through $U$. By Proposition \ref{semaffspaceassocprop}, we may also view $\phi$ as a morphism of semi-affinoid $K$-spaces. By assumption, the covering $(\phi^{-1}(U_i))_{i\in I}$ of $Y$ has a leaflike refinement; by Lemma \ref{veryspecialinafflem}, this refinement is affinoid. It now follows from Lemma \ref{easyrigidlem} that $U\subseteq X$ is $\sT_\rig$-admissible.

Let now $(U_i)_{i\in I}$ be a $\sT_\srig$-admissible covering of $U$ by $\sT_\srig$-admissible subsets $U_i$. We have seen that $U$ and the $U_i$ are $\sT_\rig$-admissible; we claim that the covering $(U_i)_{i\in I}$ is $\sT_\rig$-admissible as well. Again, we may work locally on $X^\r$ and thereby assume that the $K$-algebra of functions on $X$ is affinoid. Let $Y$ be an affinoid $K$-space, and let $\phi\colon Y\rightarrow X^\r$ be a morphism of affinoid $K$-spaces, which we may also view as a morphism of semi-affinoid $K$-spaces. Since $(U_i)_{i\in I}$ is $\sT_\srig$-admissible, we see by Proposition \ref{explicitprop} ($ii$) that $(\phi^{-1}(U_i))_{i\in I}$ has a leaflike and, hence, affinoid refinement. It follows that $(U_i)_{i\in I}$ is $\sT_\rig$-admissible.
\end{proof}

If $U\subseteq X$ is a semi-affinoid subdomain, then the morphism $U^\r\rightarrow X^\r$ provided by Proposition \ref{semaffspaceassocprop} is an open immersion onto the preimage of $U$ under the continuous bijection $\comp_X:X^\r\rightarrow X$; hence $\comp_X$ extends to a morphism of G-ringed $K$-spaces with respect to $\sT_\aux$, which then again extends uniquely to a morphism of G-ringed $K$-spaces with respect to $\sT_\urig$. One easily verifies that $\comp_X$ is local.

\begin{prop}\label{compmorunivprop}\index{rigid space!of a uniformly rigid space}
The morphism $\comp_X$ is final among all morphisms from rigid $K$-spaces to $X$.
\end{prop}
\begin{proof}
Let $Y$ be a rigid $K$-space, and let $\psi\colon Y\rightarrow X$ be a morphism of locally G-ringed $K$-spaces. By Proposition \ref{semaffspaceassocprop}, there is a unique morphism $\psi^\r\colon Y\rightarrow X^\r$ such that $\psi$ and $\comp_X\circ\psi^\r$ coincide on global sections. Since the points and the completed stalks of $X$ are recovered from the $K$-algebra of global sections of $X$, it follows that $\psi$ and $\comp_X\circ\psi^\r$ coincide.
\end{proof}

Let $X$ be any uniformly rigid $K$-space. Since the open semi-affinoid subspaces of $X$ form a basis for the G-topology on $X$, we can use standard glueing arguments to show that the comparison morphisms attached to these open semi-affinoid subspaces glue to a universal comparison morphism
\[
\comp_X\colon X^\r\rightarrow X
\]
from a rigid $K$-space to $X$.

\begin{remark}
The functor $X\mapsto X^\r$ is faithful, yet not fully faithful. For example, it is easily seen that an unbounded function on the rigid open unit disc induces a morphism to the rigid projective line over $K$ which is not induced by a morphism from the semi-affinoid open unit disc $\sSp (R[[S]]\otimes_RK)$ to the uniformly rigid projective line over $K$. Likewise, the functor $\r$ forgets the distinction between the semi-affinoid open unit disc just mentioned and the uniformly rigid open unit disc that is the generic fiber of a quasi-paracompact formal $R$-model of locally tf type for the rigid open unit disc. One can prove that $X\mapsto X^\r$ is fully faithful on the full subcategory of reduced semi-affinoid $K$-spaces.
\end{remark}

\begin{remark}
The functor $X\mapsto X^\r$ preserves fibered products. Indeed, this may be checked in the semi-affinoid situation, where it follows from the fact that fibered products of semi-affinoid spaces are uniformly rigid generic fibers of fibered products of affine flat formal $R$-models, together with the fact that Berthelot's generic fiber functor preserves fibered products, cf.\ \cite{dejong_crystalline} 7.2.4 (g). In particular, $X\mapsto X^\r$ preserves group structures.
\end{remark}

\begin{remark}
We have seen that $\comp_X$ induces isomorphisms of completed stalks. Examining Berthelot's construction, one easily sees that $\comp_X$ in fact already induces isomorphisms of non-completed stalks; the proof of this statement is left as an exercise.
\end{remark}

We have seen that every uniformly rigid $K$-space $X$ has an underlying classical rigid $K$-space $X^\r$ such that $X$ and $X^\r$ share all local properties. That is, a uniformly rigid $K$-space can be seen as a rigid $K$-space equipped with an additional global uniform structure. Every quasi-paracompact and quasi-separated rigid $K$-space carries a canonical uniformly rigid structure, which may be called the Raynaud-type uniform structure: let $\fC$ temporarily denote the category of quasi-paracompact flat formal $R$-schemes of locally tf type, and let $\fC_\Bl$ denote its localization with respect to the class of admissible formal blowups. It follows easily from the definitions that the functor $\srig|_\fC\colon\fC\rightarrow\sRig_K$ factors through a functor $\sr'\colon\fC_\Bl\rightarrow\sRig_K$. By \cite{bosch_frgnotes} Theorem 2.8/3, the functor $\rig$ induces an equivalence $\rig_\Bl$ between $\fC_\Bl$ and the category $\Rig'_K$ of quasi-paracompact and quasi-separated rigid $K$-spaces. The functor $\rig_\Bl$ will be called the \emph{Raynaud equivalence}. Composing $\sr'$ with a quasi-inverse of $\rig_\Bl$, we obtain a functor $\sr\colon\Rig'_K\rightarrow\sRig_K$; if $Y$ is in $\Rig'_K$, we say that $Y^\sr\mathrel{\mathop:}=\sr(Y)$ is the \emph{uniformly rigid $K$-space associated to $Y$}\index{uniformly rigid space!of a rigid space}. Of course, it depends on the choice of a quasi-inverse of the Raynaud equivalence.


\begin{prop}\label{srigspaceqisoidprop}
The composite functor $\r\circ\sr$ is quasi-isomorphic to the identity on $\Rig'_K$.
\end{prop}
\begin{proof} Let $\rig_\Bl^{-1}$ denote the chosen inverse of the Raynaud equivalence. Let $Y$ be an object of $\Rig'_K$; then $\rig_\Bl^{-1}(Y)$ is a quasi-paracompact flat formal $R$-model of locally tf type for $Y$, and $Y^\sr=\rig_\Bl^{-1}(Y)^\srig$, which implies that $(Y^\sr)^\r=\rig_\Bl^{-1}(Y)^\rig$, functorially in $Y$. That is, $\r\circ\sr=\rig\circ \rig_\Bl^{-1}$, which is isomorphic to the identity functor.
\end{proof}

In particular, after choosing an isomorphism $\r\circ\sr\cong\id$, the comparison morphisms $\comp_{Y^\sr}$ induce functorial comparison morphisms
\[
\comp_Y\colon Y\cong (Y^\sr)^r\rightarrow Y^\sr
\]
for all quasi-paracompact and quasi-separated rigid $K$-spaces $Y$.

\begin{cor}\label{assocsrigspacecompcor}
For $Y\in \Rig'_K$, the morphism $\comp_Y$ is the initial morphism from $Y$ to a uniformly rigid $K$-space.
\end{cor}
\begin{proof}
Let $X$ be a uniformly rigid $K$-space, and let $\psi\colon Y\rightarrow X$ be a morphism of locally G-ringed $K$-spaces. The morphism $\comp_Y$ is a bijection on points, and it induces isomorphisms of stalks; hence the morphism $Y^\sr\rightarrow X$ that we seek is unique if it exists. If $Y$ is affinoid and $X$ is semi-affinoid, there is nothing to show. Let $(X_i)_{i\in I}$ be an admissible semi-affinoid covering of $X$, and let $(Y_j)_{j\in J}$ be an admissible affinoid covering of $Y$ refining $(\psi^{-1}(X_i))_{i\in I}$. It suffices to see that $(Y_j^\sr)_{j\in J}$ is an admissible covering of $Y^\sr$. By \cite{bosch_frgnotes} Lemma 2.8/4, there exists a flat quasi-paracompact $R$-model of locally tf type $\fY$ for $Y$ such that $(Y_j)_{j\in J}$ is induced by an open covering of $\fY$. Since $\fY^\srig=Y^\sr$, it follows that $(Y_j^\sr)_{j\in J}$ is an admissible covering of $Y^\sr$, as desired. 
\end{proof}

\begin{cor}\label{ffcor}
The functor $\sr$ is fully faithful.
\end{cor}
\begin{proof}
Let $X$ and $Y$ be objects in $\Rig'_K$. By Proposition \ref{srigspaceqisoidprop},  by the global variant of Proposition \ref{compmorunivprop} and by Corollary \ref{assocsrigspacecompcor}, we have functorial bijections
\begin{eqnarray*}
\Hom(Y,X)&\cong&\Hom(Y,(X^\sr)^\r)\\
&\cong&\Hom(Y,X^\sr)\\
&\cong&\Hom(Y^\sr,X^\sr)\;.
\end{eqnarray*}
\end{proof}

Of course, if $X$ is any uniformly rigid $K$-space, then the comparison morphism 
\[
\comp_X\colon X^\r\rightarrow X
\]
is \emph{not} initial all morphisms from $X^\r$ to uniformly rigid $K$-spaces. For example, if $X$ is the semi-affinoid open unit disc $\sSp (R[[S]]\otimes_RK)$, then the natural morphism $\comp_{X^\r}$ from the rigid open unit disc $X^\r$ to its uniform rigidification $(X^\r)^\sr$ does not extend to a morphism $X\rightarrow (X^\r)^\sr$. Indeed, such a morphism would have to be the identity on points, but $X$ is quasi-compact, while $(X^\r)^\sr$ is not quasi-compact.

The functor $Y\mapsto Y^\sr$ does \emph{not} respect arbitrary open immersions. For example, if $Y'\subseteq Y$ is the inclusion of the open rigid unit disc into the closed rigid unit disc, the morphism $(Y')^\sr\rightarrow Y^\sr$ is not an open immersion: its image is the semi-affinoid open unit disc, while $(Y')^\sr$ is not quasi-compact. However, it follows from \cite{frg2} 5.7 that $\sr$ preserves open immersions of \emph{quasi-compact} rigid $K$-spaces.

Quasi-separated rigid $K$-spaces are obtained from affinoid $K$-spaces by glueing along \emph{quasi-compact} admissible open subspaces, it thus follows that $\sr$ preserves fibered products. Indeed, this can now be checked in an affinoid situation, where the statement is clear from the construction of semi-affinoid fibered products. In particular, $Y\mapsto Y^\sr$ preserves group structures.

\section{Coherent modules on uniformly rigid spaces}\label{cohmodsec}

Let $X$ be a G-ringed $K$-space, and let $\sF$ be an $\O_X$-module. Let us recall some standard definitions concerning the coherence property, cf.\ \cite{bosch_frgnotes} 1.14/2:
\begin{packed_enum}\index{module!coherent}
\item $\sF$ is called of \emph{finite type} if there exists an admissible covering $(X_i)_{i\in I}$ of $X$ together with exact sequences
\[
\O_X^{s_i}|_{X_i}\rightarrow\sF|_{X_i}\rightarrow 0\;.
\]
\item $\sF$ is called \emph{coherent} if $\sF$ is of finite type and if for any admissible open subspace $U\subseteq X$, the kernel of any morphism $\O_X^s|_U\rightarrow\sF|_U$ is of finite type.
\end{packed_enum}

If $X$ is a semi-affinoid $K$-space with ring of functions $A$, then the functor $M\mapsto \tilde{M}$ on the category of finite $A$-modules is well-behaved, as it is shown by the following lemma. The proof of Lemma \ref{assocbasicpropertieslem} is identical to the proof of \cite{bosch_frgnotes} 1.14/1; one uses the fact that the restriction homomorphisms associated to semi-affinoid subdomains are flat:

\begin{lem}\label{assocbasicpropertieslem}\index{module!associated}
The functor $M\mapsto\tilde{M}$ from the category of finite $A$-modules to the category of $\O_X$-modules is fully faithful, and it commutes with the formation of kernels, images, cokernels and tensor products. Moreover, a sequence of finite $A$-modules
\[
0\rightarrow M'\rightarrow M\rightarrow M''\rightarrow 0
\]
is exact if and only if the associated sequence
\[
0\rightarrow \tilde{M}'\rightarrow \tilde{M}\rightarrow \tilde{M}''\rightarrow 0
\]
of $\O_X$-modules is exact.
\end{lem}

For a semi-affinoid $K$-space $X=\sSp A$, we have $\O_X^r=A^r\otimes\O_X$. Since $A$ is \emph{noetherian}, it follows from Lemma \ref{assocbasicpropertieslem} that kernels and cokernels of morphisms of type $\O_X^r\rightarrow\O_X^s$ are  associated. We thus conclude that an $\O_X$-module on a uniformly rigid $K$-space $X$ is coherent if and only if there exists an admissible semi-affinoid covering $(X_i)_{i\in I}$ of $X$ such that $\sF|_{X_i}$ is associated for all $i\in I$.

In particular, the structural sheaf $\O_X$ of any uniformly rigid $K$-space $X$ is coherent. Moreover, it follows from Lemma \ref{assocbasicpropertieslem} that kernels and cokernels of  morphisms of coherent $\O_X$-modules are coherent.

\begin{lem}
Let $\phi\colon Y\rightarrow X$ be a morphism of uniformly rigid $K$-spaces, and let $\sF$ be a coherent $\O_X$-module. Then $\phi^*\sF$ is a coherent $\O_Y$-module. 
\end{lem}
\begin{proof}
Indeed, we may assume that $X$ and $Y$ are semi-affinoid, $X=\sSp A$, $Y=\sSp B$, and that $\sF$ is associated to a finite $A$-module $M$. Then $\phi^*\sF$ is associated to $M\otimes_AB$, where $B$ is an $A$-algebra via $\phi^*$.
\end{proof}

\begin{defi}
Let $X$ be a uniformly rigid $K$-space. An $\O_X$-module $\sF$ is called \emph{strictly coherent}\index{module!strictly coherent} if for any open semi-affinoid subspace $U\subseteq X$, the restriction $\sF|_U$ is an associated module.
\end{defi}

For example, the structural sheaf of a uniformly rigid $K$-space is strictly coherent. Since we do not know whether an open semi-affinoid subspace of a semi-affinoid $K$-space is a semi-affinoid subdomain, it is not a priori clear whether any associated module on a semi-affinoid $K$-space is strictly coherent. In Corollary \ref{frameembassoccor}, however, we will show that this is indeed the case.

Let $X$ be a uniformly rigid $K$-space. We will be interested in coherent $\O_X$-modules $\sF$ with the property that there exists an injective $\O_X$-homomorphism $\sF\hookrightarrow\O_X^r$ for some $r\in\N$. This property is clearly satisfied by coherent ideals, and it is preserved under pullback with respect to flat morphisms of uniformly rigid spaces. We will study integral models of such $\sF$, and we will show that any such $\sF$ is strictly coherent.

If $\fX$ is a formal $R$-scheme of locally ff type and if $\ul{\sF}$ is a coherent $\O_\fX$-module, we obtain a coherent $\O_X$-module $\ul{\sF}^\srig$ on $\fX^\srig$ which we call the \emph{uniformly rigid generic fiber} of $\ul{\sF}$. If $X$ is a uniformly rigid $K$-space, if $\sF$ is a coherent $\O_X$-module and if $\fX$ is a flat formal $R$-model of locally ff type for $X$, then an $R$-\emph{model}\index{model!of a coherent module} of $\sF$ on $\fX$ is a coherent $\O_\fX$-module $\ul{\sF}$ together with an isomorphism $\ul{\sF}^\srig\cong\sF$ that is compatible with the given identification $\fX^\srig\cong X$. Sometimes we will not mention the isomorphism $\ul{\sF}^\srig\cong\sF$ explicitly. Clearly
\[
\sp_{\fX,*}(\sF)\,=\,\ul{\sF}\otimes_RK\;,
\]
and $\srig$ factors naturally through the functor $\ul{\sF}\mapsto\ul{\sF}\otimes_RK$. Let us abbreviate $\ul{\sF}_K:=\ul{\sF}\otimes_RK$.

For any $r\in\N$, the coherent $\O_X$-module $\O_X^r$ admits the natural model $\O_\fX^r$ on every flat formal $R$-model of locally ff type $\fX$ for $X$. We will show that coherent submodules $\sF\subseteq\O_X^r$ inherit this property by taking schematic closures. Let us first consider the affine situation: 

\begin{lem}\label{submodulecanmodellem}
Let $\ul{A}$ be an $R$-algebra, let $\ul{M}$ be an $\ul{A}$-module, and let $N\subseteq \ul{M}\otimes_RK$ be an $\ul{A}\otimes_RK$-submodule. Then there exists a unique $\ul{A}$-submodule $\ul{N}\subseteq\ul{M}$ such that the natural homomorphism $\ul{N}\otimes_RK\rightarrow\ul{M}\otimes_RK$ is an isomorphism onto $N$ and such that $\ul{M}/\ul{N}$ is $R$-flat.
\end{lem}
\begin{proof}
Let us abbreviate $M\mathrel{\mathop:}=\ul{M}\otimes_RK$, and let us set
\[
\ul{N}\,\mathrel{\mathop:}=\,\ker(\ul{M}\rightarrow M/N)\;;
\]
then $\ul{N}$ is an $\ul{A}$-submodule of $\ul{M}$. For any $n\in N$, there exists an $s\in\N$ such that $\pi^sn$ lies in the image of $\ul{M}$ in $M$; the natural $K$-homomorphism $\ul{N}\otimes_RK\rightarrow N$ is thus bijective. As an $\ul{A}$-submodule of $M/N$, the quotient $\ul{M}/\ul{N}$ is free of $\pi$-torsion and, hence, $R$-flat. 

If $\ul{N}'\subseteq\ul{M}$ is another $\ul{A}$-submodule whose image in $M$ generates $N$ as an $\ul{A}\otimes_RK$-module, then $\ul{N}'$ lies in the kernel $\ul{N}$ of $\ul{M}\rightarrow M/N$. If in addition $\ul{M}/\ul{N}'$ is flat over $R$, then the natural homomorphism $\ul{M}/\ul{N}'\rightarrow \ul{M}/\ul{N}'\otimes_RK=M/N$ is injective, which proves that $\ul{N}'$ coincides with this kernel.
\end{proof}

\begin{thm}\label{frameembthm}
Let $X$ be a uniformly rigid $K$-space, let $\sF'\subseteq\sF$ be an inclusion of coherent $\O_X$-modules, and let $\fX$ be an $R$-model of locally ff type for $X$ such that $\sF$ admits be an $R$-model  $\ul{\sF}$ on $\fX$. Then there exists a unique coherent $\O_\fX$-submodule $\ul{\sF}'\subseteq\ul{\sF}$ such that $\ul{\sF}/\ul{\sF}'$ is $R$-flat and such that the given isomorphism $\ul{\sF}^\srig\cong\sF$ identifies $(\ul{\sF}')^\srig$ with $\sF'$.
\end{thm}
\begin{proof}
We may work locally on $\fX$ and thereby assume that $\fX$ is affine. Uniqueness of $\ul{\sF}'$ is a consequence of Lemma \ref{submodulecanmodellem}. Since $\sF'$ is coherent, there exists a treelike covering $(X_i)_{i\in I}$ of $X$ such that $\sF'|_{X_i}$ is associated for all $i\in\leaves(I)$. Let us choose a model of this covering, that is,
\begin{enumerate}
\item for each $i\in I$, an affine flat $R$-model of ff type $\fX_i$ for $X_i$,
\item for each inner $i\in I$ an admissible blowup $\beta_i:\fX_i'\rightarrow\fX_i$ and
\item for each inner $i\in I$ and for each child $j$ of $i$ an open immersion  $\phi_j:\fX_j\hookrightarrow\fX_i'$
such that $\fX_j\subseteq\fX_i'\rightarrow\fX_i$ represents $X_j$ in $X_i$.
\end{enumerate}
For each $i\in I$, we let $\ul{\sF}|_{\fX_i}$ denote the pullback of $\ul{\sF}$ to $\fX_i$, and for each inner vertex $i\in I$, we let $\ul{\sF}|_{\fX'_i}$ denote the pullback of $\ul{\sF}$ to $\fX_i'$. Let $i$ be an inner vertex of $I$, and let us \emph{assume} that for each child $j$ of $i$, we are given a coherent submodule
\[
\ul{\sF}'_j\,\subseteq\,\ul{\sF}|_{\fX_j}
\]
such that $\ul{\sF}|_{\fX_j}/\ul{\sF}_j'$ is $R$-flat and such that $(\ul{\sF}'_j)^\urig=\sF'|_{X_j}$. By Lemma \ref{submodulecanmodellem}, this assumption is satisfied if all children of $i$ are leaves in $I$. Using the uniqueness assertion in Lemma \ref{submodulecanmodellem}, we see that the $\ul{\sF}'_j$ glue to a unique coherent submodule 
\[
\ul{\sG}_i\subseteq\ul{\sF}|_{\fX'_i}\;.
\]
The quotient $\ul{\sF}|_{\fX'_i}/\ul{\sG}_i$ is $R$-flat; by \cite{egaiii} 3.4.2, $\beta_{i*}(\ul{\sF}|_{\fX'_i}/\ul{\sG}_i)$ thus is a co\-he\-rent $R$-flat $\O_{\fX_i}$-module. By definition, $\ul{\sF}|_{\fX'_i}=\beta_i^*\ul{\sF}|_{\fX_i}$, so we have a natural homomorphism of coherent $\O_{\fX_j}$-modules
\[
\ul{\sF}|_{\fX_i}\rightarrow \beta_{i*}\ul{\sF}|_{\fX'_i}\rightarrow\beta_{i*}(\ul{\sF}|_{\fX_i'}/\ul{\sG}_i)\;.
\]
Let $\ul{\sF}'_i$ denote its kernel; the resulting exact sequence of coherent $\O_{\fX_i}$-modules
\[
0\rightarrow \ul{\sF}'_i\rightarrow\ul{\sF}|_{\fX_i}\rightarrow\beta_{i*}(\ul{\sF}|_{\fX_i'}/\ul{\sG}_i)
\]
shows that $\ul{\sF}|_{\fX_i}/\ul{\sF}'_i$ is $R$-flat. We claim that the coherent $X_i$-module $\sF'|_{X_i}=\ul{\sG}_i^\srig$ is associated to $\ul{\sF}_i'$. To prove this, it suffices to show that the morphism
\[
(\beta_i^*\ul{\sF}_i')_K\rightarrow(\beta_i^*\beta_{i*}\ul{\sG}_i)_K\rightarrow\ul{\sG}_{i,K}\quad(*)
\]
induced by the natural morphism $\ul{\sF}'_i\rightarrow\beta_{i*}\ul{\sG}_i$ is an isomorphism. By the ff type variant of \cite{luetkebohmertformalrigid} 2.1, the second morphism in $(*)$ is an isomorphism, so we must show that the first morphism is an isomorphism as well. Let $\ul{X}_i$ be the spectrum of the ring of global functions on $\fX_i$, and let $b_i\colon\ul{X}_i'\rightarrow\ul{X}_i$ be the admissible blowup such that $\beta_i=b_i^\wedge$, where we use a wedge to denote the formal completion with respect to an ideal of definition of $\fX$. Let $\ul{F}_i$, $\ul{F}_i'$ and $\ul{G}_i$ denote the algebraizations of $\ul{\sF}|_{\fX_i}$, $\ul{\sF}_i'$ and $\ul{\sG}_i$ respectively, which exist by \cite{egaiii} 5.1.4; then 
\[
\ul{\sF}|_{\fX'_i}=(b_j^*\ul{F}_i)^\wedge\;. 
\]
By \cite{egaiii} 4.1.5,
\[
\beta_{i*}(\ul{\sF}|_{\fX_i'}/\ul{\sG}_i)\,=\,(b_{i*}((b^*_i\ul{F}_i)/\ul{G}_i)))^\wedge\;,
\]
so we have a short exact sequence
\[
0\rightarrow\ul{F}'_i\rightarrow\ul{F}_i\rightarrow b_{i*}((b^*_i\ul{F}_i)/\ul{G}_i))
\]
which under $\cdot\otimes_RK$ induces a short exact sequence
\[
0\rightarrow\ul{F}'_{i,K}\rightarrow\ul{F}_{i,K}\rightarrow (b_{i,K})_*((b^*_{i,K}\ul{F}_{i,K})/\ul{G}_{i,K}))\;.
\]
Since $b_{i,K}$ is an isomorphism and, hence, flat, we obtain an induced short exact sequence
\[
0\rightarrow b^*_{i,K}\ul{F}'_{i,K}\rightarrow b^*_{i,K}\ul{F}_{i,K}\rightarrow b^*_{i,K}(b_{i,K})_*((b^*_{i,K}\ul{F}_{i,K})/\ul{G}_{i,K}))\;;
\]
since $b^*_{i,K}(b_{i,K})_*$ is naturally isomorphic to the identity functor, this shows that $b^*_{i,K}\ul{F}'_{i,K}=\ul{G}_{i,K}$. Hence, the natural morphism
\[
b_i^*\ul{F}'_j\rightarrow b_i^* b_{i*}\ul{G}_j
\]
becomes an isomorphism under $\cdot\otimes_RK$. That is, its kernel and cokernel are $\pi$-torsion. It follows that kernel and cokernel of the completed morphism
\[
\beta_i^*\ul{\sF}_i'\rightarrow \beta_i^*\beta_{i*}\ul{\sG}_i
\]
are $\pi$-torsion as well, which yields our claim. 

Let us now prove the statement of the proposition by induction on the volume $v(I)$ of $I$. We may assume that $I$ has more than one vertex. Let $j$ be a leaf of $I$ whose path to the root has maximal length, and let $i$ be the parent of $j$. Then all children of $i$ are leaves of $I$, so the assumption in the argument above is satisfied. By what we have shown so far, $\sF'|_{X_i}$ is associated to a unique coherent $\O_{\fX_i}$-submodule $\ul{\sF}'_i\subseteq\ul{\sF}|_{\fX_i}$ such that $\ul{\sF}|_{\fX_i}/\ul{\sF}'_i$ is $R$-flat. We may thus replace $\subtree(i)$ by $\{i\}$. By induction on $v(I)$, the desired statement follows.
\end{proof}

\begin{cor}\label{frameembassoccor}
We conclude:
\begin{packed_enum}
\item A coherent submodule of an associated module on a semi-affinoid $K$-space is associated.
\item Coherent submodules and coherent quotients of strictly coherent mo\-dules are strictly coherent.
\item An associated module on a semi-affinoid $K$-space is strictly coherent.
\end{packed_enum}
\end{cor}
\begin{proof}
Let us first show ($i$). Let $X=\sSp A$ be a semi-affinoid $K$-space, let $\ul{A}\subseteq A$ be an $R$-model of ff type, and let $\sF'$ be a coherent submodule of an associated module $\tilde{M}$. Since $\tilde{M}$ admits a model $\ul{M}$ over $\Spf \ul{A}$, Theorem \ref{frameembthm} implies that $\sF'\cong(\ul{\sF}')^\srig$ for a coherent module $\ul{\sF}'$ on $\Spf\ul{A}$. Since coherent modules on affine formal schemes are associated, it follows that $\sF'$ is associated.

Let us prove statement ($ii$). Let $X$ be a uniformly rigid $K$-space, let $\sF$ be a strictly coherent $\O_X$-module and let $\sF'\subseteq\sF$ be a coherent submodule. For every open semi-affinoid subspace  $U\subseteq X$, the restriction $\sF'|_U$ is a coherent submodule of $\sF|_U$, and $\sF|_U$ is associated by assumption on $\sF$. It follows from ($i$) that $\sF'|_U$ is associated; hence $\sF'$ is strictly coherent. Let now $\sF''$ be a coherent quotient of $\sF$. Then the kernel $\sF'$ of the projection $\sF\rightarrow \sF''$ is a coherent submodule of $\sF$ and, hence, strictly coherent by what we have seen so far. Let $U\subseteq X$ be an open semi-affinoid subspace; then we have a short exact sequence
\[
0\rightarrow\sF'|_U\rightarrow\sF|_U\rightarrow\sF''|_U\rightarrow 0
\]
where the first two modules are associated. It follows from Lemma \ref{assocbasicpropertieslem} that $\sF''|_U$ is associated as well.

Finally, statement ($iii$) follows from statement ($ii$) because by Lemma \ref{assocbasicpropertieslem}, an associated module is a quotient of a finite power of the structural sheaf.
\end{proof}

If $\fX$ is a flat formal $R$-scheme of locally ff type and if $\ul{\sF}$ is a coherent $\O_\fX$-module, we do not know in general whether $\ul{\sF}^\srig$ is strictly coherent. In particular, we unfortunately do not know whether the analog of Kiehl's Theorem \cite{kiehlab} 1.2 holds in general, that is to say whether every coherent module on a semi-affinoid $K$-space is associated. Let us point out that the analogous question for \emph{quasi-}coherent modules on \emph{rigid} spaces was open for a long time; it was finally settled in the negative by O.\ Gabber, cf.\ \cite{conrad_ampleness} Example 2.1.6.
\begin{conjecture} 
The general uniformly rigid analog of Kiehl's theorem does not hold. 
\end{conjecture}

\begin{remark}
The general uniformly rigid analog of Kiehl's theorem is equivalent to the following statement: let $\fX$ be an admissible blowup of a flat affine formal $R$-scheme of ff type, and let $\sF$ be a coherent sheaf on $X=\fX^\urig$ that admits flat models $\ul{\sF}_i$ locally with respect to an affine open covering $(\fX_i)_{i\in I}$ of $\fX$; then $\sF$ admits a model on $\fX$. Indeed, this equivalence follows by arguing as in the proof of \cite{luetkebohmertformalrigid} Theorem 2.3. However, it seems impossible in general to modify the models $\ul{\sF}_i$ such that they glue to a model of $\sF$ on $\fX$: Let us assume that $I=\{1,2\}$. After multiplying $\ul{\sF}_1$ by a suitable power of $\pi$, we may assume that $\ul{\sF}_1$ is contained in $\ul{\sF}_2$ on the intersection $\fX_{12}$ of $\fX_1$ and $\fX_2$. Let $n\in\N$ be big enough such that $\pi^n\ul{\sF}_2\subseteq\ul{\sF}_1$ on $\fX_{12}$;  then $\ul{\mathcal{G}}:=\ul{\sF}_1|_{\fX_{12}}/\pi^n\ul{\sF}_2|_{\fX_{12}}$ is a coherent subsheaf of $(\ul{\sF}_2/\pi^n\ul{\sF}_2)|_{\fX_{12}}$, cf.\ the proof of \cite{luetkebohmertformalrigid} Lemma 2.2. If $\fX$ is of tf type over $R$, then the closed formal subscheme of $\fX$ cut out by $\pi^n$ is a scheme, and by chasing denominators (cf.\ \cite{egai} 9.4.7) one can extend $\ul{\mathcal{G}}$ to a coherent subsheaf, again denoted by $\ul{\mathcal{G}}$, on all of $\fX_2$. Let $\ul{\sF}_2'$ denote the preimage of $\ul{\mathcal{G}}$ under the projection $\ul{\sF}_2\rightarrow\ul{\sF}_2/\pi^n\ul{\sF}_2$; then $\ul{\sF}_2'$ is a model of $\sF$ on $\fX_2$ which glues to $\ul{\sF}_1$, and we obtain a model of $\sF$ on all of $\fX$. In our situation, however, $\fX_2$ might not be of tf type, and hence the closed formal subscheme of $\fX_2$ cut out by $\pi^n$ might not be a scheme. On a formal scheme though it is in general not possible to extend coherent subsheaves because of convergence problems. Thus, Lütkebohmert's proof of Kiehl's theorem fails in the uniformly rigid situation. Similar problems occur if one tries to carry over Kiehl's original proof.
\end{remark}

\subsection{Closed uniformly rigid subspaces}

\begin{defi}\label{closedimdefi}
A morphism of uniformly rigid $K$-spaces $\phi\colon Y\rightarrow X$ is called a \emph{closed immersion}\index{closed immersion!of uniformly rigid spaces} if there exists an admissible semi-affinoid covering $(X_i)_{i\in I}$ of $X$ such that for each $i\in I$, the restriction $\phi^{-1}(X_i)\rightarrow X_i$ of $\phi$ is a closed immersion of semi-affinoid $K$-spaces in the sense of Definition \ref{semaffclosedimdefi}.
\end{defi}

We easily see that closed immersions are injective on the level of physical points. 

\begin{lem}\label{closedimstrictcohlem}
Let $\phi\colon Y\rightarrow X$ be a closed immersion of uniformly rigid $K$-spaces. Then $\phi^\sharp\colon \O_X\rightarrow\phi_*\O_Y$ is an epimorphism of sheaves. Moreover, the $\O_X$-modules $\phi_*\O_Y$ and $\ker\phi^\sharp$ are strictly coherent. 
\end{lem}
\begin{proof}
The $\O_X$-module $\O_X$ is strictly coherent. By Corollary \ref{frameembassoccor} ($ii$), it thus suffices to show that $\phi^\sharp$ is an epimorphism and that both $\ker\phi^\sharp$ and $\phi_*\O_Y$ are coherent. Considering an admissible semi-affinoid covering $(X_i)_{i\in I}$ of $X$ such that for all $i\in I$, the restriction $\phi^{-1}(X_i)\rightarrow X_i$ of $\phi$ is a closed immersion of semi-affinoid $K$-spaces, we reduce to the case where both $X$ and $Y$ are semi-affinoid and where $\phi$ corresponds to a surjective homomorphism of semi-affinoid $K$-algebras. Now the desired statements follow from Lemma \ref{assocbasicpropertieslem}.
\end{proof}

\begin{prop}\label{closedimprop}
Let $\phi\colon Y\rightarrow X$ be a morphism of uniformly rigid $K$-spaces. Then the following are equivalent:
\begin{packed_enum}
\item $\phi$ is a closed immersion.
\item For each open semi-affinoid subspace $U\subseteq X$, the restriction $\phi^{-1}(U)\rightarrow U$ is a closed immersion of semi-affinoid $K$-spaces in the sense of Definition \ref{semaffclosedimdefi}.
\end{packed_enum}
\end{prop}
\begin{proof}
The implication ($ii$)$\Rightarrow$($i$) is trivial, the open semi-affinoid subspaces forming a basis for the G-topology on $X$. Let us assume that ($i$) holds, let $\sI$ denote the kernel of $\phi^\sharp$, and let $U\subseteq X$ be an open semi-affinoid subspace; then $\phi$ induces a short exact sequence
\[
0\rightarrow\sI|_U\rightarrow\O_U\rightarrow\phi_*\O_Y|_U\rightarrow 0\;.
\]
Let $A$ denote the ring of functions on $U$. By Lemma \ref{closedimstrictcohlem}, $\sI$ and $\phi_*\O_Y$ are strictly coherent; hence the above short exact sequence is associated to a short exact sequence of $A$-modules
\[
0\rightarrow I\rightarrow A\rightarrow B\rightarrow 0\;.
\]
Since morphisms from uniformly rigid $K$-spaces to semi-affinoid $K$-spaces correspond to $K$-homomorphisms of rings of global functions, we can now mimic the proof of \cite{bgr} 9.4.4/1 to see that the restriction $\phi^{-1}(U)\rightarrow U$ of $\phi$ is associated to the projection $A\rightarrow B$: it suffices to see that the natural morphism $\phi^{-1}(U)\rightarrow\sSp B$ is an isomorphism. This can be checked locally on $\sSp B$ with respect to the preimage under $\sSp B\rightarrow U$ of a leaflike refinement of $(U\cap X_i)_{i\in I}$, where $(X_i)_{i\in I}$ is an admissible semi-affinoid covering of $X$ satisfying the conditions of Definition \ref{closedimdefi}.
\end{proof}

\begin{remark}
The proof of \cite{bgr} 9.4.4/1 uses \cite{bgr} 8.2.1/4. However, as our above argument shows, this reference to \cite{bgr} 8.2.1/4 is in fact unnecessary -- which is to our advantage, because the statement of \cite{bgr} 8.2.1/4 fails to hold in the semi-affinoid situation: Example \ref{nonadmdisccovex} yields a bijective morphism of semi-affinoid $K$-spaces which induces isomorphisms of stalks and which is not an isomorphism.
\end{remark}

In particular, a morphism of semi-affinoid $K$-spaces is a closed immersion in the sense of Definition \ref{closedimdefi} if and only if it is a closed immersion of semi-affinoid $K$-spaces in the sense of Definition \ref{semaffclosedimdefi}. We can now define a \emph{closed uniformly rigid subspace}\index{subspace!closed uniformly rigid} as an equivalence class of closed immersions, in the usual way. By standard glueing arguments, we see that the closed uniformly rigid subspaces of a uniformly rigid $K$-space $X$ correspond to the coherent $\O_X$-ideals. We easily see that closed immersions of uniformly rigid $K$-spaces are preserved under base change.

It is clear that closed immersions of formal $R$-schemes of locally ff type induce closed immersions on uniformly rigid generic fibers. Conversely, given a uniformly rigid $K$-space $X$ together with an $R$-model of locally ff type $\fX$ and a closed uniformly rigid subspace $V\subseteq X$, there exists a unique $R$-flat closed formal subscheme $\fV\subseteq\fX$ such that the given isomorphism $\fX^\srig\cong X$ identifies $\fV^\srig$ with $V$. Indeed, this is an immediate consequence of Theorem \ref{frameembthm}. We say that $\fV$ is the \emph{schematic closure}\index{schematic closure!of a closed subspace} of $V$ in $\fX$.

The comparison functors studied in Section \ref{compfuncsec} preserve closed immersions. This can be verified in the semi-affinoid and affinoid situations respectively. In the case of the functor $\sr$, there is nothing to show. In the case of the functor $\r$, the statement follows by looking at schematic closures and using the fact that Berthelot's construction preserves closed immersions, cf.\ \cite{dejong_crystalline} 7.2.4 (e).

\subsubsection{Separated uniformly rigid spaces}\label{separatedsection}

As usual, a morphism $\phi\colon Y\rightarrow X$ of uniformly rigid $K$-spaces is called \emph{separated} if its \emph{diagonal morphism} 
\[
\Delta_\phi\colon Y\rightarrow Y\times_XY
\]
is a closed immersion. A uniformly rigid $K$-space $X$ is called separated\index{uniformly rigid space!separated} if its structural morphism $X\rightarrow\sSp K$ is separated. If $X$ is a uniformly rigid $K$-space, we let $\Delta_X$ denote the diagonal of its structural morphism.

Semi-affinoid $K$-spaces are visibly separated. Moreover, uniformly rigid generic fibers of separated morphisms of formal $R$-schemes of locally ff type are separated, since functor $\srig$ preserves fibered products and closed immersions. Similarly, the comparison functors studied in Section \ref{compfuncsec} preserve the separatedness property.

\begin{lem}\label{sepintlem}
Let $X$ be a separated uniformly rigid $K$-space. The intersection of two open semi-affinoid subspaces in $X$ is an open semi-affinoid subspace in $X$.
\end{lem}
\begin{proof}
Let $U$ and $V$ be open semi-affinoid subspaces in $X$. We easily see, using points with values in finite field extensions of $K$, that $U\cap V$ is the $\Delta_X$-preimage of $U\times_{\sSp K}V$ which is an open semi-affinoid subspace of $X\times_{\sSp K}X$. Since $\Delta_X$ is a closed immersion by assumption on $X$, it follows from Proposition \ref{closedimprop} that $U\cap V$ is an open semi-affinoid subspace of $X$.
\end{proof}

\begin{cor}\label{calcthecohcor}
Let $X$ be a separated uniformly rigid $K$-space, and let $\sF$ be a coherent $\O_X$-module. Then the natural morphism
\[
\check{H}^q(X,\sF)\overset{\sim}{\rightarrow}H^q(X,\sF)
\]
is an isomorphism for all $q\geq 0$.
\end{cor}
\begin{proof}
Let $S$ denote the set of open semi-affinoid subspaces $U$ in $X$ with the property that $\sF|_U$ is associated. By Lemma \ref{sepintlem}, this set is stable under the formation of intersections. It is clearly a basis for the G-topology on $X$, and $\check{H}^q(U,\sF)=0$ for any $U$ in $S$ and any $q\geq 0$ by Corollary \ref{moduleacythmcor}. We conclude by the usual \v{C}ech spectral sequence argument.
\end{proof}

If $X$ is a separated uniformly rigid $K$-space and if $\phi\colon Y\rightarrow X$ is a morphism of uniformly rigid $K$-spaces, then the graph $\Gamma_\phi\colon Y\rightarrow Y\times X$ of $\phi$ is a closed immersion since it is obtained from $\Delta_X$ via pullback. In particular, if $\fX$ and $\fY$ are $R$-models of locally ff type for $X$ and $Y$ respectively, the schematic closure of $\Gamma_\phi$ in $\fY\times\fX$ is well-defined. Here fibered products without indication of the base are understood over $\sSp K$ or $\Spf R$ respectively.

\section{Comparison with the theories of Berkovich and Huber}\label{berthhubercompsec}

The category of formal $R$-schemes of locally ff type is a full subcategory of Huber's category of adic spaces, cf.\ \cite{huberbook}. If $\fX$ is a formal $R$-scheme of locally ff type, viewed as an adic space, then by \cite{huberbook} 1.2.2 the fibered product $\fX\times_{\Spa(R,R)}\Spa(K,R)$ is the adic space associated to the rigid generic fiber $\fX^\rig$ of $\fX$. That is, the uniform structure induced by $\fX$ is lost. In fact, we do not see a way to view the category of uniformly rigid spaces as a full subcategory of Huber's category of adic spaces. The main obstacle lies in the fact that if $\ul{A}$ is an $R$-algebra of ff type, equipped with its natural Jacobson-adic topology, and if $A=\ul{A}\otimes_RK$, then the pair $(A,\ul{A})$ is in general \emph{not} an f-adic ring in the sense of \cite{huberbook}. For example, for $\ul{A}=R[[S]]$ there exists no ring topology on $A$ such that $\ul{A}$ is open in this topology: There is a unique such group topology, but multiplication by $\pi^{-1}$ in $A$ is not continuous because there is no $n\in\N$ such that $\pi^{-1}S^n\in R[[S]]\otimes_RK$ is contained in $R[[S]]$.

The situation is different if we consider the $\pi$-adic topology on $R$-algebras of ff type. If $\ul{A}^\pi$ denotes the ring $\ul{A}$ equipped with its $\pi$-adic topology, then the pair $(A,\ul{A}^\pi)$ is an f-adic ring in the sense of Huber. The induced topology on $A$ is in fact a $K$-Banach algebra topology; if, for $f\in A$ nonzero, we set $v_{\ul{A}}(F)\,:=\,\max\{n\in\N\,;\,\pi^{-n}f\in\ul{A}\}$, then $|f|_{\ul{A}}\,:=\,|\pi|^{v_{\ul{A}}(f)}$  defines a $K$-Banach algebra norm on $A$ which induces the topology defined by $\ul{A}^\pi$. If $\ul{A}=R[[S]]\langle T\rangle$ is a mixed formal power series ring in finitely many variables, then $|\cdot|_{R[[S]]\langle T\rangle}$ is the Gauss norm, and it coincides with the supremum semi-norm taken over all points in $\Max A$. Using \cite{bgr} 3.7.5/2, one proves that all $K$-Banach algebra structures on $A$ are equivalent; in particular, the valuation spectrum $M(A)$ in the sense of \cite{berkobook} 1.2 is well-defined. One shows that reduced semi-affinoid $K$-algebras are Banach function algebras, and one verifies that the supremum semi-norm, taken over all points in $\Max A$ or, equivalently, over all points in $M(A)$, takes values in $\sqrt{|K|}$. For a more detailed discussion, including proofs, we refer to Section 1.2.5 in the author's PhD thesis \cite{mythesis}.

The topological space $M(A)$ may be viewed as a compactification of the rigid space $(\sSp A)^\r$. To illustrate this idea in terms of an example, let us first explain how the specialization map extends to valuation spectra. If $A$ is a semi-affinoid $K$-algebra and if $\ul{A}$ is an $R$-model of ff type for $A$, there exists a natural specialization map 
\[
\sp_{\ul{A}}\colon M(A)\rightarrow\Specns(\ul{A}/\pi\ul{A})
\]
extending the specialization map which we discussed in Section \ref{specmapsec}: let $x$ be a point in $M(A)$, represented by a character $\chi_x:A\rightarrow\sK$ with values in some valued field extension $\sK$ of $K$; then $\sp_{\ul{A}}(x):=\ker(\tilde{\chi}_x:\ul{A}/\pi\ul{A}\rightarrow\tilde{\sK})$, where $\tilde{\sK}$ is the residue field of $\sK$ and where $\tilde{\chi}_x$ is the reduction of $\chi_x$. 

\begin{lem}\label{specmapberklem}
The map $\sp_{\ul{A}}$ is surjective onto $\Specns(\ul{A}/\pi\ul{A})$. Moreover, if $\ul{A}/\pi\ul{A}$ is a domain, then the residue norm $|\cdot|_{\ul{A}}$ is multiplicative and, hence, defines a point in $M(A)$. This point specializes to the generic point of $\Specns(\ul{A}/\pi\ul{A})$, and it is the only point in $M(A)$ with this property.
\end{lem}
\begin{proof}
Surjectivity of $\sp_{\ul{A}}$ follows from \cite{egaii} 7.1.7. If $\ul{A}/\pi\ul{A}$ is a domain, then $|\cdot|_{\ul{A}}\in M(A)$ clearly specializes to $\pi\ul{A}$. Moreover, the local ring $\ul{A}_{\pi\ul{A}}$ is then a discrete valuation ring, such that every character $\chi$ of a point $x\in M(A)$ specializing to the generic point of $\Spec\ul{A}/\pi\ul{A}$ is equivalent to the character given by the natural homomorphism from $A$ to the fraction field of the $\pi$-adic completion of $\ul{A}_{\pi\ul{A}}$. It follows that $x$ equals $|\cdot|_{\ul{A}}$.
\end{proof}

One can easily verify that when $\ul{A}/\pi\ul{A}$ is a domain, then $\{|\cdot|_{\ul{A}}\}$ is the Shilov boundary of $M(A)$, cf.\ \cite{mythesis} 1.2.5.12; we will not use this fact in the following. Let us now discuss the example of the open unit disc $\sSp (R[[S]]\otimes_RK)$:

\begin{example}\label{discexample}
The set $M(R[[S]]\otimes_RK)$ is naturally identified with the closure of the Berkovich open unit disc within $M(K\langle S\rangle)$, which is obtained by adding the Gauss point.
\end{example}

\begin{proof}
To understand the continuous map $i\colon M(R[[S]]\otimes_RK)\rightarrow M(K\langle S\rangle)$ induced by the natural isometry $K\langle S\rangle\hookrightarrow R[[S]]\otimes_RK$, we distinguish the points in $M(R[[S]]\otimes_RK)$ with respect to their specializations to the scheme $\Spec k[[S]]$.
Applying Lemma \ref{specmapberklem} to $\ul{A}=R[[S]]$, we see that the unique point above the generic point of $\Spec k[[S]]$ is the Gauss point $|\cdot|_\Gauss$, which maps to the Gauss point in $M(K\langle S\rangle)$ via $i$. If $x\in M(R[[S]]\otimes_RK)$ is a point specializing to the special point of $\Spec k[[S]]$, then for any character $\chi_x$ representing $x$, the induced $R$-homomorphism $\mathring{\chi}_x\colon R[[S]]\rightarrow\mathring{\sK}$ is continuous for the $(\pi,S)$-adic topology on $R[[S]]$ and the valuation topology on $\mathring{\sK}$. In particular, $\chi_x$ is determined by the $\chi_x$-image of the variable $S$. We conclude that the map $i$ is injective and that it maps the complement of the Gauss point onto the Berkovich open unit disc.
The image of $i$ is the continuous image of a compact set and, hence, compact. Since $M(K\langle S\rangle)$ is Hausdorff, it follows that the image of $i$ is closed in $M(K\langle S\rangle)$. 
\end{proof} 
 
\begin{remark}
Given a complete non-trivially valued non-archimedean field $K$ with valuation ring $R$, one may wonder whether the points of the rigid open unit disc over $K$ lie dense in $M(R[[S]]\otimes_RK)$; this question is called the one-dimensional non-archimedean Corona problem. It is yet unanswered; cf.\ the introduction of \cite{deninger_corona} for a brief survey including other versions of non-archimedean Corona problems. If $K$ is discretely valued (which is the overall assumption in this paper), our discussion of Example \ref{discexample} above shows that the Corona question has a positive answer: indeed, let $Z\subseteq M(R[[S]]\otimes_RK)$ be the closure of the set of classical points; then the image of $Z$ under the natural map $i$ to the $K$-analytic space $M(K\langle S\rangle)$ is closed. Working locally on $M(K\langle S\rangle)$, we see that $i(Z)$ contains the Berkovich open unit disc and, hence, its closure. We have seen in Example \ref{discexample} that $i$ is injective onto that closure; thus it follows that $Z=M(R[[S]]\otimes_RK)$. The one-dimensional non-archimedean Corona problem is significantly more challenging when $K$ is not discretely valued: then the ring $R[[S]]\otimes_RK$ is not Noetherian, it has maximal ideals of infinite height (cf.\ \cite{vanderput_corona} Corollary 4.9), and it contains functions with infinitely many zeros on the rigid open unit disc.
\end{remark} 
 
It is natural to ask whether one can associate a topological space to a uniformly rigid $K$-space such that, in the semi-affinoid case, one recovers the construction $\sSp A\mapsto M(A)$ which we described above. However, the formation of $M(A)$ does not behave well with respect to localization; cf.\ the following example. This is not surprising: the Banach $K$-algebra structure on $A$ restricts to the $\pi$-adic topology on an $R$-model of ff type $\ul{A}$ for $A$, and complete localization of $\ul{A}$ with respect to the $\pi$-adic topology does in general not agree with complete localization with respect to the topology defined by the Jacobson radical. Similarly, the extended specialization map $\sp_{\ul{A}}$ maps onto the algebraization $\Specns(\ul{A}/\pi\ul{A})$ of the special fiber $\Spfns(\ul{A}/\pi\ul{A})$ of $\Spf\ul{A}$ whose formation, again, does in general not commute with localization.

\begin{example}\label{berkospecloc}
If $\ul{A}=R\langle X,Y\rangle[[Z]]/(XY-Z)$, equipped with the Jacobson-adic topology, and if $\ul{B}=\ul{A}_{\{X-Y\}}$, then the induced map $M(B)\rightarrow M(A)$ is not injective.
\end{example}
\begin{proof}
Let us write $\fX:=\Spf\ul{A}$, and let $\fX_0:=\Spec k[X,Y]/(XY)$ denote the smallest subscheme of definition of $\fX$. Since $\fX$ is formally smooth over $R$, its special fiber $\fX_k$ is formally smooth over $k$. The underlying topological space $|\fX_k|=|\fX_0|$ is connected; hence the ring $\ul{A}/\pi\ul{A}$ is a domain. By Lemma \ref{specmapberklem}, there exists a unique point $|\cdot|_{\ul{A}}$ of $M(A)$ specializing to the generic point of the algebraization  $\fX_k^\pi:=\Specns(\ul{A}/\pi\ul{A})$ of the special fiber $\fX_k=\Spf (\ul{A}/\pi\ul{A})$ of $\fX$. On the other hand, let us consider the open formal subscheme $\fU:=\Spf \ul{B}$ of $\fX$. Its underlying smallest subscheme of definition $\fU_0$ is
\[
\fU_0=\Spec (k[X,Y]/(XY))_{X-Y}\,=\,\Specns(k[X,X^{-1}]) \amalg \Specns(k[Y,Y^{-1}])\;,
\]
so $\fU$ has exactly two connected components. We conclude that $\ul{B}$ is a nontrivial direct sum $\ul{B}_1\oplus\ul{B}_2$ of flat $R$-algebras of ff type. Since $\fU$ is formally $R$-smooth, we see that $\ul{B}_i/\pi\ul{B}_i$ is a domain for $i=1,2$. We obtain an induced nontrivial decomposition $B=B_1\oplus B_2$ and, hence, a nontrivial decomposition $M(B)=M(B_1)\amalg M(B_2)$. By the proof of the statement in Example \ref{discexample}, there exist unique elements $|\cdot|_{\ul{B}_i}\in M(B_i)$, $i=1,2$, specializing to the respective generic point of $\fU_{i,k}^\pi:=\Spec\ul{B}_i/\pi\ul{B}_i$. To prove that the natural map $M(B)\rightarrow M(A)$ is not injective, it suffices to see that it maps the elements $|\cdot|_{\ul{B}_1}$, $|\cdot|_{\ul{B}_2}$ in $M(B)$ to $|\cdot|_{\ul{A}}$. By functoriality of the specialization map, it thus suffices to observe that the natural morphism $\fU_{i,k}^\pi\rightarrow\fX_k^\pi$ maps the generic point to the generic point. However, this is clear because $\ul{A}/\pi\ul{A}\rightarrow\ul{B}_i/\pi\ul{B}_i$ is injective. Indeed, it is a flat homomorphism of domains, where flatness follows from the fact that $\fU_{i,k}\rightarrow\fX_k$ is an open immersion of formal schemes.
\end{proof}

In the light of Example \ref{berkospecloc}, it is unclear how to define a global analog of $M(A)$. Nonetheless, we think that a quasi-compact uniformly rigid $K$-space $X$ should be viewed as a compactification of its underlying rigid $K$-space $X^\r$. This should be made more precise by studying the topos of $X$. 
\bibliographystyle{plain}
\bibliography{uniformlyrigidspacesbib}

\end{document}